\documentclass[a4paper]{amsart} 
\usepackage{amsmath, amsthm, amssymb}
\usepackage[margin=2.5cm]{geometry}
\usepackage[english]{babel}
\usepackage{amsfonts}
\usepackage{mathtools}
\usepackage{nccmath}
\usepackage{stackengine}
\usepackage{setspace}
\usepackage[T1]{fontenc}
\usepackage{bm}
\usepackage{relsize}
\usepackage{pgfplots}
\usepackage{accents}
\usepackage{enumitem}
\usepackage{scalerel}
\usepackage{cancel}
\usepackage{systeme}
\usepackage{stackrel}
\usepackage{xfrac}
\usepackage{extarrows}
\usepackage{graphicx}
\graphicspath{{pictures/}}
\usepackage{blkarray}
\usepackage{bigstrut}
\usepackage{braket}
\usepackage{setspace}
\usepackage{cases}
\usepackage{ctable} 

\def \lptext {u}
\newcommand{\Ae}{a\text{.}e\text{.}\ }
\usepackage[mathscr]{euscript}
\usepackage{tikz}
\usetikzlibrary{patterns}
\usetikzlibrary{arrows.meta}
\usepackage{romanbar}
\theoremstyle{plain}
\newtheorem{thm}{Theorem}[section]

\newtheorem{lemma}[thm]{Lemma}

\newtheorem{claim}[thm]{Claim}
\theoremstyle{definition}
\newtheorem{deff}[thm]{Definition}

\newtheorem*{note}{Note to the reader}
\theoremstyle{remark}
\newtheorem{rmk}[thm]{Remark}

\theoremstyle{plain}
\newtheorem{thm*}{Theorem}
\newtheorem{lemma*}{Lemma}

\DeclareMathOperator{\Div}{\mathrm{div}}
\DeclareMathOperator{\N}{\mathbb{N}}

\DeclareMathOperator{\R}{\mathbb{R}}

\DeclareMathOperator{\ee}{\mathrm{e}}

\newcommand{\HH}{H^1_0}

\newcommand{\dex}{\,\mathrm{d}x}

\newcommand{\Def}{\vcentcolon=}

\newcommand{\ssubset}{\subset\joinrel\subset}

\newcommand{\mathtitle}[1]{\texorpdfstring{#1}{\lptext{}}}
\newcommand{\RNum}[1]{\uppercase\expandafter{\romannumeral #1\relax}}

\pgfplotsset{compat=1.18}
\title{Regularizing effects for an elliptic system of singular equations}
\author{Gabriele Giannone}
\address{Dipartimento di Matematica e Informatica “U. Dini”, Università di Firenze, Viale Morgagni 67/A, 50134 Firenze, Italy, e-mail: gabriele.giannone@unifi.it}
\usepackage[autostyle,italian=guillemets]{csquotes}
\usepackage[style=alphabetic,backend=biber]{biblatex}
\usepackage[a-1b]{pdfx}
{\left\lbrace\begin{array}{@{}l@{}}}%
{\end{array}\right.}
\usepackage[a-1b]{pdfx}
\begin{document}
\begin{abstract}
A system of two singular semi--linear elliptic equations, patterned after the Schr\"odinger-Maxwell system, is considered. If the reaction term of the first equation contains a datum $f\in L^m$, existence of positive solutions with finite energy is established for suitable ranges of $m$. In particular, the results from the theory of elliptic single equations are improved. At the same time, thanks to an approach based on approximation schemes and a priori estimates on the approximated sequences of solutions, it is shown that the integrability assumptions on the datum produce higher integrability of the solutions.
\end{abstract}
\maketitle
\section{Introduction}
Let $\Omega$ be a bounded domain in $\R^d$, $d\ge3$.
In this paper, we are concerned with the doubly singular semi--linear system
\begin{equation}\label{second_system}
\begin{cases}
\displaystyle -\Div(A(x) D  u)+v^{1-\theta} u^{r-1}=\frac{f(x)}{u^\gamma} &\mbox{in}\ \Omega\\
\displaystyle-\Div(A(x) D  v)=\frac{u^r}{v^{\theta}} &\mbox{in}\ \Omega\\
v>0,\, u\ge0 &\mbox{in}\ \Omega\\
u>0 &\mbox{in}\ \{f>0\}\\
u=v=0 &\mbox{on}\ \partial\Omega.
\end{cases}
\end{equation}
Here, $f$ is a non-negative and non-trivial function belonging to a suitable Lebesgue space, and $A(\cdot)\in L^\infty(\Omega;\R^{d\times d})$ is a symmetric and uniformly elliptic matrix, in the sense that
\begin{equation}\label{ellipticity}
\alpha|\xi|^2\le A(x)\xi\cdot\xi\le\beta|\xi|^2
\end{equation}
for \Ae $x\in\Omega$ and every $\xi\in\R^d$, with $0<\alpha\le\beta$. Furthermore, we make the following assumption on the parameters:
\begin{equation}\label{main assumption on parameters}
r\ge2,\quad \gamma\in(0,1),\quad {\theta}\in[0,1).
\end{equation}
We examine how the structure of the system and the integrability of the datum $f$ give rise to regularizing effects on the existence of solutions $(u,v)\in\HH(\Omega)\times\HH(\Omega)$ having higher integrability than that prescribed by the theory of elliptic (single) equations.
\subsection{State of the art} System \eqref{second_system} is patterned after the following Schr\"odinger-Maxwell system
\begin{equation}
\begin{cases}\label{Benci_Fortunato_system}
-\frac{1}{2}\Delta u+\phi u=\omega u\\
-\Delta \phi= 4\pi u^2
\end{cases}
\end{equation}
that appeared first in the seminal work \cite{Benci_Fortunato_1998}, where V. Benci and D. Fortunato studied the eigenvalue problem for the Schr\"odinger operator when coupled with an electromagnetic field.\\
In our case, the problem is \emph{doubly singular}, namely, it is a Dirichlet problem where both equations are singular. The literature about singular problems is wide and well established. Without the intention to be exhaustive, we recall some of the most known works in the field, referring the reader to the surveys \cite{guarnotta2022,Guarnotta2023,oliva2024singularellipticpdesextensive} for more information.

First, we focus on problems with just one equation. In \cite{Lazer1991OnAS, BoccardoOrsina2010} the authors studied existence of solutions (both with finite and infinite energy) to equations modelled by
$$\begin{cases}
\displaystyle-\Delta u=\frac{f}{u^\gamma} &\mbox{in}\ \Omega\\
u=0  &\mbox{on}\ \partial\Omega.
\end{cases}$$
In particular, \cite{Lazer1991OnAS} proved the existence of a unique solution $0<u\in C^{2,\alpha}(\Omega)\cap C(\overline{\Omega})$, provided that $\partial\Omega\in C^{2,\alpha}$ and $0<f\in C^{\alpha}(\overline{\Omega})$; moreover, it was shown that $u$ belongs to $H^1_0(\Omega)$ iff $\gamma<3$. On the other hand, \cite{BoccardoOrsina2010} dealt with existence, uniqueness and regularity results of distributional solutions, when the non-negative datum $f$ is either an $m$-integrable function or a Radon measure concentrated on a Borel set of zero capacity. The authors proved that the unique solution $u$ belongs to $H^1_0(\Omega)$ both when $\gamma<1$ and $f\in L^{(\frac{2^\star}{1-\gamma})'}(\Omega)$, and when $\gamma=1$ and $f\in L^1(\Omega)$.

Paper \cite{OlivaPetitta2018} considered Dirichlet problems of the form
$$\begin{cases}
\displaystyle-\Div(A(x) Du)=h(u)\mu &\mbox{in}\ \Omega,\\
u=0  &\mbox{on}\ \partial\Omega,
\end{cases}$$
assuming that $\Omega$ is smooth, $A$ has Lipschitz continuous coefficients, and $\mu$ is a Radon measure (not necessarily concentrated on a set of zero harmonic capacity). Here, $h(s)$ is a non-negative function defined on $[0,+\infty)$, controlled from above by $s^{-\gamma}$ near zero. Both for $\gamma\le1$ and for $\gamma>1$, the authors proved that the problem has a solution in a suitable distributional sense, which is unique when $h$ is decreasing. In particular, when $h$ is controlled from above by $s^{-\theta}$ near infinity, with $\theta>0$, and $\mu$ is a non-negative measurable function $f$, the authors showed that the problem as a solution in $\HH(\Omega)$ in the cases: $\gamma\le 1$, $\theta\ge1$, and $f\in L^1(\Omega)$; $\gamma\le 1$, $\theta>0$, $f\in L^1(\Omega)\cap L^p(\Omega\smallsetminus \overline{\Omega_\varepsilon})$ and $f(x)\le \frac{C}{\delta(x)}$ \Ae in $\Omega_\varepsilon$, where $p>\frac{d}{2}$, $\delta(x)=\mathrm{dist}(x,\partial\Omega)$, and $\Omega_\varepsilon=\{x\in\Omega : \delta(x)<\varepsilon\}$; $f\in L^m(\Omega)$, with $m>1$, $1<\gamma<2-\frac{1}{m}$, and $\theta\ge1$. Moreover, when $\gamma>1$, $h(u)=u^{-\gamma}$, and $f$ belongs to $L^m(\Omega)$, $m>1$, the authors proved that the solution $u$ belongs to $\HH(\Omega)$ iff $\gamma<3-\frac{2}{m}$, thus extending the previously quoted result in \cite{Lazer1991OnAS}.

In \cite{Oliva2019}, the author proved existence of solutions to problems whose model is
$$\begin{cases}
\displaystyle-\Delta_p u+u^q=h(u)f &\mbox{in}\ \Omega,\\
u\ge0  &\mbox{in}\ \Omega,\\
u=0  &\mbox{on}\ \partial\Omega,
\end{cases}$$
where $1\le p<d$, $f$ belongs to suitable Lebesgue spaces, and $h(s)$ is a non-negative function defined on $[0,+\infty)$, controlled from above by $s^{-\gamma}$ near zero, and by $s^{-\theta}$ near infinity. The author studied the regularizing effect of the absorption term, establishing existence of finite energy solutions (i.e., $u\in W^{1,p}_0(\Omega)$ if $p>1$, and $u\in BV(\Omega)$ if $p=1$) in the cases: $\gamma\le 1$, $\theta\ge1$, and $f\in L^1(\Omega)$; $\gamma\le 1$, $\theta<1$, $f\in L^m(\Omega)$, with $m>1$, and $q\ge\frac{1-m\theta}{m-1}$.

In the last twenty years, many papers have been concerned with various kinds of problems bearing a resemblance to \eqref{Benci_Fortunato_system}. Given $B>0$, $r>1$, $0\le f\in L^s(\Omega)$, and $A$ and $M$ measurable uniformly elliptic matrices, the papers \cite{Boccardo2016SpiritofBenciandFortunato,BoccardoOrsina2016} studied several regularizing effects of elliptic semi--linear systems of the type
$$\begin{cases}
-\Div(A(x) D  u)+B v u^{r-1} =f &\mbox{in}\ \Omega\\
-\Div(M(x) D  v)=u^r &\mbox{in}\ \Omega\\
u,v\ge0 &\mbox{in}\ \Omega\\
u=v=0 &\mbox{on}\ \partial\Omega
\end{cases}$$
which, when $M=A$, correspond to \eqref{second_system} for $\gamma=\theta=0$. Under the same assumptions, in \cite{Durastanti_2019} it was proved the existence of finite-energy solutions to the quasi--linear system
$$\begin{cases}
-\Delta_p u + B v^{\tau+1}|u|^{r-2}u=f &\mbox{in}\ \Omega\\
-\Delta_p v =|u|^r v^{\tau} &\mbox{in}\ \Omega\\
v\ge0 &\mbox{in}\ \Omega\\
u=v=0 &\mbox{on}\ \partial\Omega
\end{cases}$$
where $1<p<d$ and $0\le\tau<p-1$, which for $p=2$ is case of \eqref{second_system} for $\gamma=0$ and $1-p<\theta\le0$. For $q\in\{2,p\}$, the authors considered both data $f$ in the dual space (of the first equation), that is, $s\ge (q^\star)'$, and data outside the dual space, assuming some lower bound on $s$ depending from $r$. Among the considered cases, we recall those more related to the present paper: $s>\frac{d}{q}$, with solutions in $L^\infty(\Omega)$; $s\ge (q^\star)'$, with $u\in L^{\frac{ds(q-1)}{d-ps}}$; in \cite{Boccardo2016SpiritofBenciandFortunato}, $2\le s\le \frac{d}{2r'}$, with $u\in L^s(\Omega)$; in \cite{BoccardoOrsina2016}, $r'\le s<(2^\star)'$, with $u\in L^{s(r-1)}(\Omega)$; in \cite{Durastanti_2019}, ${\tau}=0$ and $s\ge(r+1)'$, with $u\in L^\alpha(\Omega)$, $\alpha=\frac{s(qr+q-1)}{s(q-1)+1}$.

In \cite{DeCaveOlivaStrani2016}, the authors studied the existence of solutions to the following non-variational singular system
$$\begin{cases}
-\Delta_p u = a(x)v^{\rho}u^{r-1} &\mbox{in}\ \Omega\\
-\Delta_p v = b(x)u^r v^{\rho-1} &\mbox{in}\ \Omega\\
u,v>0 &\mbox{in}\ \Omega\\
u=v=0 &\mbox{on}\ \partial\Omega
\end{cases}$$
which, disregarding the term $fu^{-\gamma}$, is case of \eqref{second_system} for $a=b=1$ and $\theta-1=-\rho$. In this non-variational setting, under the assumptions $0\le a\in L^\infty(\Omega)$, $0<\beta\le b\in L^m(\Omega)$, $m>\frac{d}{p}$, $0<\rho<1$, $0<r<p-\rho$, and $\rho+1<p<d$ with $p\neq1+\sqrt{r\rho}$, the authors established the existence of a solution in $(W^{1,p}_0(\Omega)\cap L^\infty(\Omega))^2$. Note that this set of assumptions allows both equations to be (mildly) singular. Let us also mention that, when $a$ and $b$ are constant, and replacing the $p$-laplacian is replaced by a second-order linear uniformly elliptic operator, the problem was treated in \cite{BoccardoOrsina2011} with variational techniques.

Singular nonlinearities for Schr\"odinger-Maxwell systems were first considered in \cite{BOCCARDO2022126490}, concerned with the elliptic problem
\begin{equation*}
\begin{cases}
\displaystyle -\Div(A(x) D  u)+v u^{r-1}=\frac{1}{u^\gamma} &\mbox{in}\ \Omega\\
-\Div(M(x) D  v)=u^r &\mbox{in}\ \Omega\\
u,v>0 &\mbox{in}\ \Omega\\
u=v=0 &\mbox{on}\ \partial\Omega
\end{cases}
\end{equation*}
where both mild and strong singularities were considered. Moreover, in the mild case $\gamma\in(0,1]$ the condition $r>1-\gamma$ was imposed, allowing the nonlinear term $u^{r-1}$ to be singular as well. When $M=A$ and $\gamma\in(0,1)$, this corresponds to \eqref{second_system} for $f=1$, ${\theta}=0$.

In \cite{LiuMosconi2020689} the reduction procedure of Benci--Fortunato was exploited to study the stationary Schr\"odinger-Poisson system
\begin{equation*}
\begin{cases}
-\Delta u + V u+\phi u= |u|^{p-1}u\\
-\Delta \phi =u^2.
\end{cases}
\end{equation*}
It is worth to notice that, when $p=3$, the first equation above describes the dynamic of standing waves of a Bose-Einstein condensate of electrically charged particles, under the assumption that the wave function of the condensate $\psi\colon\R^3\times[0,+\infty)\to\mathbb{C}$ writes as $\psi(x,t)=\ee^{-i\omega t}u(x)$, with $u\colon\R^3\to\R$. The authors proved existence of non-trivial solutions when the potential $V$ is coercive and sign-changing, and they allowed the nonlinear term to be $3$-sublinear, filling the gap with the existing literature, where the $3$-superlinear case only was considered.

Finally, in \cite{BoccardoOrsina2024}, the authors established some regularizing effects for the singular elliptic system
\begin{equation*}
\begin{cases}
\displaystyle -\Div(A(x) D  u)+v^{1-\theta} u^{r-1}=f &\mbox{in}\ \Omega\\
\displaystyle-\Div(M(x) D  v)=\frac{u^r}{v^{\theta}} &\mbox{in}\ \Omega\\
u,v>0 &\mbox{in}\ \Omega\\
u=v=0 &\mbox{on}\ \partial\Omega
\end{cases}
\end{equation*}
where $0<\theta<1$, which for $M=A$ corresponds to \eqref{second_system} for $\gamma=0$.
\subsection{Organization of the paper} The paper is organized as follows: in Section \ref{main results}, we summarize our main results; in Section \ref{notations and preliminaries} we set some notation and we recall some basic results in elliptic theory; in Section \ref{Approximation scheme} existence of an approximating sequence of solutions is established studying an approximation of \eqref{second_system}; Sections \ref{regular data}, \ref{data dual space}, \ref{data outside dual space}, \ref{A higher integrability result} are then devoted to the proof of the main results.
\section{Main results}\label{main results}
In what follows, unless explicitly stated, let $\Omega\subset\R^d$ be an open and bounded domain, let \eqref{ellipticity} and \eqref{main assumption on parameters} be in force and let $f$ belong to $L^m(\Omega)$ for some $m\ge1$. Moreover, given $1\le p\le\infty$, we denote by $p'$ and $p^\star$ the H\"older and the Sobolev conjugates of $p$, respectively, namely
$$p'\Def\begin{cases}
\displaystyle\frac{p}{p-1} &\mbox{if}\ 1<p<\infty\\
\infty &\mbox{if}\ p=1\\
1 &\mbox{if}\ p=\infty
\end{cases}\qquad p^\star\Def\begin{cases}
\displaystyle\frac{dp}{d-p} &\mbox{if}\ p<d\\
\text{any}\ q\in[1,\infty) &\mbox{if}\ p=d\\
\infty &\mbox{if}\ p>d.
\end{cases}$$

First, we address the case $m\ge\frac{d}{2}$. Let us precisely define what kind of solutions we are interested in.
\begin{deff}[Weak solutions]\label{def weak solution}
We say that $(u,v)\in\HH(\Omega)\times\HH(\Omega)$ is a \textit{weak solution} to system \eqref{second_system} if $u,v>0$ \Ae in $\Omega$,
\begin{gather}\label{integrability of reactions weak solutions}
\frac{f}{u^\gamma}\varphi,\,\frac{u^r}{v^{\theta}}\varphi\in L^1(\Omega)\quad\forall\varphi\in\HH(\Omega),
\end{gather}
and
$$\begin{cases}
\displaystyle\int\limits_\Omega A D  u\cdot D  \varphi+ v^{1-\theta} u^{r-1}\varphi\dex=\int\limits_\Omega \frac{f}{u^\gamma}\varphi\dex&\ \ \forall\varphi\in \HH(\Omega)\\[10pt]
\displaystyle\int\limits_\Omega A D  v\cdot D  \psi\dex=\int\limits_\Omega \frac{u^r}{v^{\theta}}\psi\dex&\ \ \forall\psi\in \HH(\Omega).
\end{cases}$$
\end{deff}
The main result of Section \ref{regular data} is the following
\begin{thm}\label{existence m>d/2}
If $m\ge\frac{d}{2}$, then system \eqref{second_system} has a weak solution $(u,v)$, such that $u$ belongs to $L^\infty(\Omega)$ if $m>\frac{d}{2}$, and to every $L^p(\Omega)$, $1\le p<\infty$, if $m=\frac{d}{2}$. Moreover, $v$ belongs to $L^\infty(\Omega)$.
\end{thm}
Here, the boundedness of solutions is obtained via Theorem \ref{Stampacchia_theorem} below, which we can apply thanks to the integrability assumption on $f$ (see Section \ref{overview of the proofs} for more details). When $m>\frac{d}{2}$, Theorem \ref{existence m>d/2} recovers the boundedness results in \cite{DeCaveOlivaStrani2016} for non-variational singular elliptic systems, in \cite{BoccardoOrsina2010} for singular elliptic equations, and in \cite{Boccardo2016SpiritofBenciandFortunato,Durastanti_2019,BOCCARDO2022126490,BoccardoOrsina2024} for systems of Schr\"odinger-Maxwell type where (at least) one of the two equations is non-singular.

Then, assuming $m<\frac{d}{2}$, we study some regularizing effects of the system, as done in \cite{Boccardo2016SpiritofBenciandFortunato,BoccardoOrsina2016,BoccardoOrsina2018,Durastanti_2019} for non-singular problems. With this aim, in what follows, if $u>0$ \Ae in $\{f>0\}$, we define
\begin{equation}\label{definition of function g}
g(x,u)=\begin{cases}
\displaystyle\frac{f(x)}{u^{\gamma}} &\mbox{in}\ \{f>0\}\\
0 &\mbox{otherwise}.
\end{cases}
\end{equation}
Clearly, $g(\cdot,u)=fu^{-\gamma}$ iff $u>0$ \Ae in $\Omega$.\\
To prove our results, we have to restrict the space of test functions for the first equation. More precisely, we give the following notion of solution.
\begin{deff}[Finite-energy solutions]\label{def finite-energy solutions}
We say that $(u,v)\in\HH(\Omega)\times\HH(\Omega)$ is a \textit{finite-energy solution} to system \eqref{second_system} if $v>0$ \Ae in $\Omega$, $u>0$ \Ae in $\{f>0\},
$\begin{gather}\label{integrability of reactions finite-energy solutions}
g(\cdot,u)\varphi\in L^1(\Omega)\ \forall\varphi\in\HH(\Omega)\cap L^\infty(\Omega)\qquad \text{and}\qquad\frac{u^r}{v^{\theta}}\psi\in L^1(\Omega)\ \forall\psi\in\HH(\Omega),
\end{gather}
and
$$\begin{cases}
\displaystyle\int\limits_\Omega A D  u\cdot D  \varphi+ v^{1-\theta} u^{r-1}\varphi\dex=\int\limits_\Omega g(\cdot,u)\varphi\dex&\ \ \forall\varphi\in \HH(\Omega)\cap L^\infty(\Omega)\\[10pt]
\displaystyle\int\limits_\Omega A D  v\cdot D  \psi\dex=\int\limits_\Omega \frac{u^r}{v^{\theta}}\psi\dex&\ \ \forall\psi\in \HH(\Omega).
\end{cases}$$
\end{deff}
\begin{rmk}
As it will be evident in the proofs of the main results, this definition will allow us to deal with the first equation even when \emph{no strong maximum principle} holds, that is, for all choices of $v$, $\theta$, and $r$.
\end{rmk}
At first, we consider the case $\big(\frac{2^\star}{1-\gamma}\big)'\le m < \frac{d}{2}$, when $f$ belongs to the \emph{dual space} of the first equation $L^{(\frac{2^\star}{1-\gamma})'}(\Omega)$. In Section \ref{data dual space} we prove:
\begin{thm}\label{existence data in dual space}
Let $\big(\frac{2^\star}{1-\gamma}\big)'\le m<\frac{d}{2}$. Then system \eqref{second_system} has a finite-energy solution $(u,v)$, with $u\in L^{m^{\star\star}(1+\gamma)}(\Omega)$. Moreover, if $r=2$, one has that $v\in L^{s_m}(\Omega)$, with
\begin{equation}\label{regularity v dual space}
s_m=\begin{cases}
\infty &\mbox{if}\ m>\frac{d}{3+\gamma},\\
\text{any } p\in[1,\infty) &\mbox{if}\ m=\frac{d}{3+\gamma},\\
\left(\medmath{\frac{m^{\star\star}(1+\gamma)}{2}}\right)^{\star\star}(1+\theta) &\mbox{if}\ \big(\frac{2^\star}{1-\gamma}\big)'\le m<\frac{d}{3+\gamma},
\end{cases}
\end{equation}
and that $u>0$ \Ae in $\Omega$ whenever
\begin{equation}\label{conditions strict positivity dual space}
\begin{cases}
\displaystyle d=3,4,5\\
\displaystyle d\ge 6\quad\text{and}\quad \theta> \frac{d-6}{d-2}\\
\displaystyle d\ge 6\quad\text{and}\quad m>\frac{1-\theta}{2(2-\theta+\gamma)}d\quad\text{and}\quad \theta\le\frac{d-6}{d-2}.
\end{cases}
\end{equation}
\end{thm}
\begin{rmk}\label{remark data dual space}
$(i)$ Let $\delta\in(0,1)$. By Theorem $5.2$ in \cite{BoccardoOrsina2010} and the strict monotonicity of $s\mapsto s^{-\delta}$ in $(0,+\infty)$, given $g\in L^p(\Omega)$, $p\ge \big(\frac{2^\star}{\delta}\big)'$, there exists a unique weak solution $w\in \HH(\Omega)$ to the Dirichlet problem
\begin{equation*}
\begin{cases}
\displaystyle-\Div(A(x) D  w)=\frac{g}{w^{\delta}} & \mbox{in}\ \Omega\\
w=0 & \mbox{on}\ \partial\Omega.
\end{cases}
\end{equation*}
Hence, the standing integrability assumption on $f$ makes the existence of weak solutions to the first (single) equation of \eqref{second_system} somehow ``not surprising''. On the other hand, since, in general, $u^r$ does not belong to $L^{(\frac{2^\star}{1-\theta})'}(\Omega)$, the existence of solutions $v\in \HH(\Omega)$ to the second equation is an unexpected feature, and it is a consequence of the structure of the problem.

$(ii)$ The conditions \eqref{conditions strict positivity dual space}, which guarantee the strict positivity of $u$ in the whole of $\Omega$ when $r=2$, may seem somehow exotic. Nevertheless, bearing in mind \eqref{regularity v dual space}, they become (at least more) natural once one recalls that, by the weak Harnack inequality in Theorem $7.1.2$ in \cite{pucci-serrin}, and the resulting strong maximum principle (Corollary $7.1.3$ in \cite{pucci-serrin}), any weak solution $0\le w\in H^1_{\text{loc}}(\Omega)$ to 
$$-\Div(A(x)Dw)+v^{1-\theta}w\ge0$$
satisfies
$$w\equiv0\quad\text{or}\quad w>0\quad \text{\Ae in}\ \Omega$$
whenever $v^{1-\theta}\in L^q(\Omega)$, with $q>\frac{d}{2}$.

$(iii)$ In light of the statements of Theorems \ref{existence m>d/2} and \ref{existence data in dual space}, we see that in the limit case $m=\frac{d}{2}$ the regularity gain on $u$ is intermediate between that for $m<\frac{d}{2}$ and that for $m>\frac{d}{2}$.
\end{rmk}
Then, we assume that $f$ does not belong to the dual space, namely, we impose $m <\big(\frac{2^\star}{1-\gamma}\big)'$. Below are the main results of Section \ref{data outside dual space}:
\begin{thm}\label{regularizing outside dual space}
If $\big(\frac{r}{1-\gamma}\big)'\le m<\big(\frac{2^\star}{1-\gamma}\big)'$ and $r>2^\star$, then system \eqref{second_system} has a finite-energy solution $(u,v)$, with $u\in L^{r}(\Omega)$.
\end{thm}
\begin{thm}\label{regularizing outside dual space 2}
If $\big(\frac{r+1}{1-\gamma}\big)'\le m<\big(\frac{2^\star}{1-\gamma}\big)'$, $r>2^\star-1$, and ${\theta}=0$, then system \eqref{second_system} has a finite-energy solution $(u,v)$, with $u\in L^{r+1}(\Omega)$.
\end{thm}
The reader may observe that the theorems above have quite similar statements. In fact, the main difference is that in Theorem \ref{regularizing outside dual space 2} the allowed range for $m$ is wider than that in Theorem \ref{regularizing outside dual space}, but the second equation becomes non-singular.\\
Let us stress that both Theorem \ref{regularizing outside dual space} and Theorem \ref{regularizing outside dual space 2} are existence and \emph{regularity} results. Indeed, in both of them the solution $u$ is proved to belong to a suitable Lebesgue space --- depending on the chosen range for $m$ --- which may be smaller than $L^{2^\star}(\Omega)$.
\begin{rmk}
In Theorems \ref{regularizing outside dual space} and \ref{regularizing outside dual space 2}, the class of data for which solutions are proved to exist in $\HH(\Omega)$ is wider than that considered for single singular equations. Indeed, as mentioned in Remark \ref{remark data dual space}, for single equations the existence of solutions with finite energy is proved only for data in the dual space, while in the present situation $m<\big(\frac{2^\star}{1-\gamma}\big)'$ and $u^r$ does not necessarily belong to $L^{(\frac{2^\star}{1-\theta})'}(\Omega)$.
\end{rmk}
Finally, provided $m\ge (r+\gamma)'>\big(\frac{r}{1-\gamma}\big)'$, we show how, thanks to the singularity of the first equation, the solution $u$ given by Theorem \ref{regularizing outside dual space} has a higher integrability than that guaranteed by the previous results. More precisely, in Section \ref{A higher integrability result} we prove the following:
\begin{thm}\label{last thm}
Let $ (r+\gamma)' \le m<\frac{d}{2}$ and let $(u,v)\in\HH(\Omega)\times\HH(\Omega)$ be the finite-energy solution to system \eqref{second_system} given by Theorem \ref{regularizing outside dual space}. If $r\ge\frac{d}{d-2}-\gamma$, then  $u\in L^{r+1+\gamma}(\Omega)$.
\end{thm}
\subsection{An overview of the proofs}\label{overview of the proofs}
Our approach is based on an approximation procedure, that shows the existence of a sequence $(u_n,v_n)$ that solves a suitable elliptic system, and on some a priori estimates on the latter, which guarantee that $(u_n,v_n)$ converges to a solution to the system having the desired integrability. More precisely, set $f_n=\min\{f,n\}$, we show that the problem
\begin{equation}\label{approx system intro}
\begin{cases}
\displaystyle-\Div(A(x) D  u_n)+v_n^{1-\theta} u_n^{r-1}=\frac{f_n}{\big(u_n+\frac{1}{n}\big)^\gamma} &\mbox{in}\ \Omega\\
\displaystyle-\Div(A(x) D  v_n)=\frac{u_n^r}{\big(v_n+\frac{1}{n}\big)^\theta} &\mbox{in}\ \Omega\\
u_n, v_n>0 &\mbox{in}\ \Omega\\
u_n=v_n=0 &\mbox{on}\ \partial\Omega
\end{cases}
\end{equation}
has a unique weak, bounded solution. To do this, first we examine the equations ``separately'', namely, we prove the existence of unique $u_z,v_\phi\in \HH(\Omega)\cap L^\infty(\Omega)$, $u_z,v_\phi>0$ \Ae in $\Omega$, that satisfy, in the weak sense,
\begin{gather*}
-\Div(A(x)D u_z)+z^{1-\theta} u_z^{r-1}=\frac{f_n}{\big(u_z+\frac{1}{n}\big)^\gamma},\\-\Div(A(x)D v_\phi)=\frac{(\min\{k,|\phi|\})^r}{\big(v_\phi+\frac{1}{n}\big)^{\theta}},
\end{gather*}
being $\phi\in\HH(\Omega)$, $0\le z\in L^\infty(\Omega)$, and $k>0$. By Theorem \ref{Stampacchia_theorem}, as a by-product of this construction $u_z$ and $v_\phi$ satisfy a priori $L^\infty$--estimates independent from $\phi$ and $z$, respectively (but, in general, \emph{depending} on $n$). Taking advantage of this estimates, we conclude proving that the map $\HH(\Omega)\ni\phi\mapsto u_{v_\phi}$ has a fixed point $(u_n,v_n)$, solution to \eqref{approx system intro} in the weak sense.

In order to establish Theorems \ref{existence m>d/2}, \ref{existence data in dual space}, \ref{regularizing outside dual space}, \ref{regularizing outside dual space 2}, we first prove suitable a priori estimates on $(u_n)$, which depends on the standing regularity assumption on $f$. In particular, we show that $(u_n)$ is bounded in $L^{p_m}(\Omega)$, where
$$p_m=\begin{cases}
\infty &\mbox{if}\ m>\frac{d}{2}\\
\text{any}\ p\in[1,\infty) &\mbox{if}\ m=\frac{d}{2}\\
m^{\star\star}(1+\gamma) &\mbox{if}\ \big(\frac{2^\star}{1-\gamma}\big)'\le m <\frac{d}{2}\\
{r} &\mbox{if}\ \big(\frac{r}{1-\gamma}\big)'\le m<\big(\frac{2^\star}{1-\gamma}\big)'\\
{r+1} &\mbox{if}\ \big(\frac{r+1}{1-\gamma}\big)'\le m<\big(\frac{2^\star}{1-\gamma}\big)'\ \text{and}\ \theta=0.\\
\end{cases}$$
Then, we deduce a priori energy-estimates on $(u_n)$ and $(v_n)$, which, in turn, imply that (a proper subsequence of) $(u_n,v_n)$ converges weakly in $\HH(\Omega)\times\HH(\Omega)$ and strongly in $L^2(\Omega)\times L^2(\Omega)$ to a weak/finite-energy solution to the system. Let us emphasize that, when dealing with finite-energy solutions, we see that $f_n\big(u_n+\frac{1}{n}\big)^{-\gamma}\to g(\cdot,u)$ \Ae in $\Omega$, where $g$ is the function defined in \eqref{definition of function g}.

Concerning the last part of Theorem \ref{existence data in dual space}, we first exploit a regularization argument (inspired by Lemma $5.5$ and Theorem $5.6$ in \cite{BoccardoOrsina2010}) to deduce some a priori estimates on $(v_n)$, and then we show that, inasmuch as \eqref{conditions strict positivity dual space} holds, the first equation satisfies a strong maximum principle. As a consequence, we see that $g(\cdot,u)=fu^{-\gamma}$ \Ae in $\Omega$.

Finally, to prove Theorem \ref{last thm}, we note that $ (r+\gamma)' \le m<\frac{d}{2}$ implies the almost everywhere convergence $(u_n,v_n)$ to a finite-energy solution $(u,v)$. The higher integrability of $u$ is then proved via an a priori $L^{r+1+\gamma}(\Omega)$--estimate on $(u_n)$.
\begin{rmk}
Let us point out that, although applications of the Sobolev embedding are ubiquitous in our proofs, we do not make any regularity assumption on $\Omega$. This is possible because, as it is well-known, the extension by zero outside $\Omega$ maps $\HH(\Omega)$ isometrically into $H^1(\R^d)$ (see Lemma $3.27$ in \cite{adamsfournier} for a proof).
\end{rmk}
\section{Notations and preliminaries}\label{notations and preliminaries}
We define the \emph{positive} and \emph{negative} \emph{part} of a real-valued function $u\colon\Omega\to\R$ as
\begin{gather*}
u_+(x)\Def \max\{u(x),0\}=\begin{cases}
u(x) &\mbox{if}\ u(x)>0\\
0 &\mbox{otherwise}
\end{cases}\\
u_-(x)\Def (-u)_+(x)=\begin{cases}
-u(x) &\mbox{if}\ u(x)<0\\
0 &\mbox{otherwise}
\end{cases}
\end{gather*}
for all $x\in\Omega$. Let $k>0$. We will make use of the following real functions
\begin{gather*}
T_k(s)\Def\max\{-k,\min\{s,k\}\}=\begin{cases}
-k &\mbox{if}\ s\le-k,\\
s &\mbox{if}\ |s|\le k,\\
k &\mbox{if}\ s\ge k,
\end{cases}\\
G_k(s)\Def (|s|-k)_+\mathrm{sign}(s)=
\begin{cases}
s-k &\mbox{if}\ s\ge k,\\
0 &\mbox{if}\ |s|\le k,\\
s+k &\mbox{if}\ s\le -k.\\
\end{cases}
\end{gather*}
\begin{rmk}
If $u\in W^{1,p}(\Omega)$, the same happens to $u_+$, $T_k(u)$, $G_k(u)$, and
$$ D  u_+= D  u\chi_{\{u\ge0\}},\quad  D  T_k(u)= D  u \chi_{\{|u|\le k\}},\quad D  G_k(u)= D  u \chi_{\{|u|\ge k\}}\quad\text{\Ae in } \Omega.$$
For a proof see, for instance, Lemma $1.1$ in \cite{stampacchia6364}, Theorem $4.2$ in \cite{Boccardomurat1982}, or Theorem $4.4$ \cite{evans-gariepy}.
\end{rmk}
\subsection{Boundedness of solutions to uniformly elliptic PDEs} Given $F\in H^{-1}(\Omega)$ and $M(\cdot)\in L^\infty(\Omega;\R^{d\times d})$ such that \eqref{ellipticity} holds true. As it is well known, there exists a unique weak solution $u\in\HH(\Omega)$ to
\begin{equation}\label{Dirichlet linear problem}
\begin{cases}
-\Div(M(x) D  u)=F &\mbox{in}\ \Omega\\
u=0 &\mbox{on}\ \partial\Omega.
\end{cases}
\end{equation}
A function $u\in\HH(\Omega)$ is a weak sub-solution to \eqref{Dirichlet linear problem} if
$$\int\limits_\Omega M D  u\cdot D \varphi\dex\le\langle F,\varphi\rangle\quad\forall\varphi\in\HH(\Omega),\, \varphi\ge0,$$
where $\langle\cdot,\cdot\rangle$ are the brackets for the duality pairing between $H^{-1}(\Omega)$ and $\HH(\Omega)$.
Since every weak-subsolution $u\in\HH(\Omega)$ to \eqref{Dirichlet linear problem} satisfies $\alpha\|u\|_{\HH}\le\|F\|_{H^{-1}}$, if one assumes that $F\in L^p(\Omega)$, $p\ge (2^\star)'$, it holds
$$\alpha\|u\|_{\HH}\le|\Omega|^{1-\frac{(2^\star)'}{p}}\|F\|_{L^p}.$$
In what follows, we will frequently make use of the following boundedness result, proved by G. Stampacchia in the seminal work \cite{stampacchia65}. For its proof, we refer the reader to Théorème $4.2$ in the aforementioned paper.
\begin{thm}\label{Stampacchia_theorem}
Suppose $F\in L^p(\Omega)$, with $p>\frac{d}{2}$, and let $u\in\HH(\Omega)$ be a weak solution to \eqref{Dirichlet linear problem}. Then $u\in L^\infty(\Omega)$, and there exists a positive constant $C$ depending only upon $d$, $|\Omega|$, $\alpha$, and $p$, such that
\begin{equation*}
\|u\|_{\HH}+\|u\|_{L^\infty}\le C\|F\|_{L^p}.
\end{equation*}
\end{thm}
Actually, the argument that proves Theorem \ref{Stampacchia_theorem} yields the following result.
\begin{thm}\label{cor stamacchia theorem}
Suppose $F\in L^p(\Omega)$, with $p>\frac{d}{2}$, and let $u\in\HH(\Omega)$ such that
$$\text{there exists }k_0\ge 0\ \text{such that}\ \int\limits_{\Omega} \alpha|D G_k(u)|^2\dex\le \int\limits_{\Omega} F G_k(u)\dex\ \ \forall k\ge k_0.$$
Then $u\in L^\infty(\Omega)$, and there exists a positive constant $C$ depending only upon $d$, $|\Omega|$, $\alpha$, and $p$, such that
\begin{equation*}
\|u\|_{\HH}\le C\|F\|_{L^p}\quad\text{and}\quad 
\|u\|_{L^\infty}\le k_0+
C\|F\|_{L^p}.
\end{equation*}
\end{thm}
\subsection{Convergence of the gradients} A Borel measure $\mu$ on $\Omega$ is called Radon measure if it is locally finite and inner regular, i.e., if $\mu(K)<+\infty$ for all $K\subset\Omega$ compact and
\begin{gather*}
\mu(U)=\sup\{\mu(K) : K\subset U\ \text{compact}\}
\end{gather*}
for all $U\subset\Omega$ open. The space of Radon measures on $\Omega$ is denoted by $\mathcal{R}(\Omega)$. As it is well known, this space coincides with the topological dual of $C_{\mathrm{c}}(\Omega)$, and we say that a sequence $(\mu_n)$ is bounded in $\mathcal{R}(\Omega)$ if, for every $K\subset\Omega$ compact,
$$|\langle\mu_n,\phi\rangle|\le c\|\phi\|_{L^\infty(K)}\quad\forall\phi\in C_{\mathrm{c}}(\Omega),\,\mathrm{supp}(\phi)\subset K,$$
for some positive constant $c=c(K)$.

When passing to the limit in our approximation scheme, we will take advantage of the following result:
\begin{thm}\label{Boccardo-Murat}
Let $M(\cdot)\in L^\infty(\Omega;\R^{d\times d})$ be a uniformly elliptic matrix and let $(f_n)\subset \mathcal{R}(\Omega)$. Suppose there exists a distributional solution $u_n\in H^1(\Omega)$ to
$$-\Div(M(x)Du_n)=f_n.$$
If $(f_n)$ is bounded in $\mathcal{R}(\Omega)$ and $u_n\rightharpoonup u$ in $H^1(\Omega)$, then
$$Du_n\to Du\quad\text{strongly in}\ L^q(\Omega; \R^d)\ \text{for any}\ 1\le q<2.$$
In particular, up to a subsequence, $Du_n\to Du$ \Ae in $\Omega$.
\end{thm}
For the proof of (a more general version of) this result, we refer the reader to Theorem $2.1$ in \cite{BOCCARDO1992581}.

\begin{note}
Throughout the paper, the letter $C$ will denote various positive constants --- whose value may change from line to line --- depending on the data of the problem ($d$, $|\Omega|$, $\alpha$, $m$, $\gamma$, etc.), but never on the sequence indexes that we shall use in the proofs.
\end{note}
\section{Approximation scheme}\label{Approximation scheme}
To begin, we provide an approximation procedure for the first equation. Fix $0\le z\in L^\infty(\Omega)$. For any $n\in\N$, define $f_n(x)\Def T_n(f(x))$ and consider the problem:
\begin{equation}
\begin{cases}\label{approximated problem 1}
\displaystyle-\Div(A(x) D  u_n)+z^{1-\theta} u_n^{r-1}=\frac{f_n}{\big(u_n+\frac{1}{n}\big)^\gamma} &\mbox{in}\ \Omega\\
u_n>0 &\mbox{in}\ \Omega\\
u_n=0 &\mbox{on}\ \partial\Omega.
\end{cases}
\end{equation}
We show that \eqref{approximated problem 1} has a unique weak solution. Our proof exploits fixed point arguments, and it avoids any variational tool.
\begin{lemma}\label{approximated solution eq 1}
For every $n\in\N$, there exists a unique $u_n\in \HH(\Omega)\cap L^\infty(\Omega)$, weak solution to \eqref{approximated problem 1}, such that $u_n>0$ \Ae in $\Omega$ and $\|u_n\|_{\HH}+\|u_n\|_{L^\infty}\le C n^{1+\gamma}$ for all $n\in\N$, for some positive constant $C=C(d,|\Omega|,\alpha)$.
\end{lemma}
\begin{proof} Denote by $\kappa=\kappa(d,|\Omega|,\alpha)>0$ the constant in Theorem \ref{Stampacchia_theorem}. We divide the proof into steps.

\emph{Step $1$: construction of $u_n$.} Fix $n\in\N$ and let $k_n=\kappa^{r-2}n^{(r-2)(1+\gamma)}$.  Define $F(\psi)\Def f_n\big(|\psi|+\frac{1}{n}\big)^{-\gamma}$ for all $\psi\in\HH(\Omega)$. Since $F(\HH(\Omega))\subset L^\infty(\Omega)\subset L^2(\Omega)$, by the Lax-Milgram theorem there exists a unique weak solution $u_\psi\in\HH(\Omega)$ to
\begin{equation}\label{claim lemma 1}
-\Div(A(x) D  u_\psi)+z^{1-\theta }T_{k_n}(|\psi|^{r-2})u_\psi =F(\psi).
\end{equation}
Testing \eqref{claim lemma 1} against $(u_\psi)_-$ we get
\begin{align*}
0\le  \int\limits_\Omega F(\psi) (u_\psi)_-\dex&= \int\limits_\Omega A D  u_\psi\cdot D  (u_\psi)_-+z^{1-\theta} T_{k_n}(|\psi|^{r-2}) u_\psi(u_\psi)_-\dex\\
&\le -\alpha\int\limits_\Omega | D (u_\psi)_-|^2\dex\le0.
\end{align*}
It follows that $(u_\psi)_-=0$ \Ae in $\Omega$, i.e., $u_\psi\ge0$ \Ae in $\Omega$. Then, letting $k\ge0$ and choosing $G_k(u_\psi)$ as a test function in \eqref{claim lemma 1} produce
$$\int\limits_\Omega \alpha |D(G_k(u_\psi))|^2\dex\le\int\limits_\Omega \frac{f_n}{\big(|\psi|+\frac{1}{n}\big)^{\gamma}} G_k(u_\psi)\dex,$$
hence, since $f_n\big(|\psi|+\frac{1}{n}\big)^{-\gamma}\in L^\infty(\Omega)$ and $\|f_n\big(|\psi|+\frac{1}{n}\big)^{-\gamma}\|_{L^\infty}\le n^{1+\gamma}$, Theorem \ref{cor stamacchia theorem} ensures that $u_\psi\in L^\infty(\Omega)$ and
\begin{equation}\label{first a priori estimates}
\|u_{\psi}\|_{\HH}+\|u_{\psi}\|_{L^\infty}\le \kappa n^{1+\gamma}.
\end{equation}
Let $S\colon\HH(\Omega)\to\HH(\Omega)$ be defined as $S(\psi)=u_{\psi}$. By \eqref{first a priori estimates}, the ball $B=\{\psi\in\HH(\Omega) : \|\psi\|_{\HH}\le \kappa n^{1+\gamma}\}$ is invariant for $S$. We want to show that $S$ has a fixed point.

\emph{Sub-step $1.1$: $S$ is weakly--strongly continuous.} Let $\psi_j\rightharpoonup \psi$ in $\HH(\Omega)$ and, without loss of generality, assume that $\psi_j\to\psi$ in $L^p(\Omega)$ for all $1\le p<2^\star$. Taking advantage of \eqref{first a priori estimates}, there exists $\xi\in\HH(\Omega)$ such that, up to a subsequence, $(u_{\psi_j})$ converges to $\xi$ weakly in $\HH(\Omega)$, strongly in $L^p(\Omega)$ for all $1\le p<2^\star$, and \Ae in $\Omega$. Choosing $u_{\psi_j}-u_{\psi}$ in the weak formulations of the problems solved by $u_{\psi_j}$ and $u_\psi$ and subtracting the equations yield
\begin{multline}\label{inequality weak formulation}
\int\limits_\Omega A  D  (u_{\psi_j}-u_\psi)\cdot D  (u_{\psi_j}-u_\psi)\dex=\int\limits_\Omega (F(\psi_j)-F(\psi))(u_{\psi_j}-u_\psi)\dex\\-\int\limits_\Omega z^{1-\theta} \big(T_{k_n}(|\psi_j|^{r-2}) u_{\psi_j}-T_{k_n}(|\psi|^{r-2}) u_\psi\big)(u_{\psi_j}-u_\psi)\dex.
\end{multline}
Since $(u_{\psi_j}-u_\psi)$ is bounded in $L^\infty(\Omega)$, by Lebesgue's dominated convergence theorem one deduces that $(F(\psi_j)-F(\psi))(u_{\psi_j}-u_\psi)\to0$ and $z^{1-\theta} T_{k_n}(|\psi_j|^{r-2}) u_{\psi_j}\to z^{1-\theta} T_{k_n}(|\psi|^{r-2}) \xi$ in $L^1(\Omega)$. Then, \eqref{inequality weak formulation} entails 
$$\limsup_{n\to\infty}\int\limits_\Omega A  D  (u_{\psi_j}-u_\psi)\cdot D  (u_{\psi_j}-u_\psi)\dex=-\int\limits_\Omega z^{1-\theta} T_{k_n}(|\psi|^{r-2})(\xi-u_\psi)^2\dex\le0,$$
whence
$$\lim_{n\to\infty}\int\limits_\Omega A  D  (u_{\psi_j}-u_\psi)\cdot D  (u_{\psi_j}-u_\psi)\dex=0.$$
By the uniform ellipticity of $A(\cdot)$,
it follows that $u_{\psi_j}\to u_\psi$ in $\HH(\Omega)$, and this is enough to conclude that $S$ is weakly--strongly continuous.

Since $S$ is compact and continuous from $B$ into $B$, by the Schauder fixed point theorem there exists $0\le u_n\in B\subset \HH(\Omega)\cap L^\infty(\Omega)$ such that $u_n=S(u_n)$, namely,
$$\int\limits_\Omega A D  u_n\cdot D \varphi+ z^{1-\theta} T_{k_n}(u_n^{r-2})u_n\varphi\dex=\int\limits_\Omega \frac{f_n}{\big(u_n+\frac{1}{n}\big)^\gamma} \varphi\dex$$
for all $\varphi\in\HH(\Omega)$. Then, to prove that $u_n$ is a weak solution to \eqref{approximated problem 1}, it suffices to note that, by \eqref{first a priori estimates}, $u_n^{r-2}\le \kappa^{r-2}n^{(r-2)(1+\gamma)}=k_n$, hence $T_{k_n}(u_n^{r-2})u_n=u_n^{r-1}$.

\emph{Step $2$: $u_n$ is positive.} We now show that $u_n>0$ \Ae in $\Omega$. Note that
$$-\Div(A(x) D u_n)+\|z\|_{L^\infty}^{1-\theta}u_n^{r-1}\ge \frac{f_n}{\big(\kappa n^{1+\gamma}+\frac{1}{n}\big)^\gamma}\ge0.$$
Then, taking advantage of the boundedness of $u_n$, and recalling that $f_n\not\equiv0$, by the weak Harnack inequality (see Theorem $1.2$ in \cite{Trudinger1967}) we conclude that $u_n>0$ \Ae in $\Omega$.

\emph{Step $3$: uniqueness.} Finally, we observe that $u_n$ is unique. Indeed, if $U_n$ is another positive solution, one has
\begin{gather*}
\int\limits_\Omega ADu_n\cdot D(u_n-U_n)_++z^{1-\theta} u_n^{r-1}(u_n-U_n)_+\dex=\int\limits_\Omega \frac{f_n}{\big(u_n+\frac{1}{n}\big)^\gamma}(u_n-U_n)_+\dex\\
\int\limits_\Omega ADU_n\cdot D(u_n-U_n)_++z^{1-\theta} U_n^{r-1}(u_n-U_n)_+\dex=\int\limits_\Omega \frac{f_n}{\big(U_n+\frac{1}{n}\big)^\gamma}(u_n-U_n)_+\dex,
\end{gather*}
whence, taking the difference,
\begin{multline*}
\alpha\int\limits_\Omega |D(u_n-U_n)_+|^2\dex\le \int\limits_\Omega z^{1-\theta}(U_n^{r-1}-u_n^{r-1})(u_n-U_n)_+\dex\\
+ \int\limits_\Omega f_n\bigg(\frac{1}{\big(u_n+\frac{1}{n}\big)^\gamma}-\frac{1}{\big(U_n+\frac{1}{n}\big)^\gamma}\bigg)(u_n-U_n)_+\dex.
\end{multline*}
Since the maps $s\mapsto (s+1/n)^{-\gamma}$, $s\mapsto s^{r-1}$ are increasing in $[0,+\infty)$, the previous inequality entails
$$0\le\alpha\int\limits_\Omega |D(u_n-U_n)_+|^2\dex\le0,$$
hence $(u_n-U_n)_+=0$, i.e., $u_n\le U_n$. Likewise, $U_n\le u_n$. Therefore, $U_n=u_n$, and the solution is unique.
\end{proof}
\begin{lemma}(Approximation)\label{approximation lemma}
For every $n\in\N$, there exists $(u_n,v_n)\in\big(\HH(\Omega)\cap L^\infty(\Omega)\big)^2$ such that $u_n,v_n>0$ \Ae in $\Omega$ and
\begin{numcases}{}
$$\displaystyle\int\limits_\Omega A D  u_n\cdot D  \varphi+ v_n^{1-\theta} u_n^{r-1}\varphi\dex=\int\limits_\Omega\frac{f_n}{\big(u_n+\frac{1}{n}\big)^\gamma}\varphi\dex$$\label{appr eq1}\\
$$\displaystyle \int\limits_\Omega A D  v_n\cdot D \psi\dex=\int\limits_\Omega\frac{u^r_n}{\big(v_n+\frac{1}{n}\big)^{\theta}}\psi\dex$$\label{appr eq2}
\end{numcases}
for all $\varphi,\psi\in\HH(\Omega)$.
\end{lemma}
\begin{proof}\let\qed\relax
Given $k\ge0$ and $\phi\in\HH(\Omega)$, by the same argument that proves Lemma \ref{approximated solution eq 1}, and by the strong maximum principle for the operator $-\Div(A(x)D \cdot)$, there exists $0\le v_\phi\in\HH(\Omega)$ that solves
$$-\Div(A(x) D  v_\phi)=\frac{T_k(|\phi|)^r}{\big(v_\phi+\frac{1}{n}\big)^{\theta}}$$
in the weak sense; this solution turns out to be unique, due to the strict monotonicity of $s\mapsto s^{{-\theta}}$ in $[0,+\infty)$. Moreover, by Theorem \ref{Stampacchia_theorem}, $v_{\phi}\in L^\infty(\Omega)$, and $\|v_\phi\|_{\HH}+\|v_\phi\|_{L^\infty}\le C k^r n^{\theta}$. By Lemma \ref{approximated solution eq 1}, there is a unique weak solution $0<u_\phi\in \HH(\Omega)\cap L^\infty(\Omega)$ to
$$-\Div(A(x) D  u_\phi)+v_\phi^{1-\theta} u_\phi^{r-1}=\frac{f_n(x)}{\big(u_\phi+\frac{1}{n}\big)^\gamma}$$
that satisfies $\|u_n\|_{L^\infty}\le Cn^{1+\gamma}$. Now, we make the following:
\vspace{-0.1cm}
\begin{claim}\label{claim compactness}
The map $S\colon\HH(\Omega)\to\HH(\Omega)$ defined by $S(\phi)=u_\phi$ is compact and continuous.
\end{claim}
\vspace{-0.1cm}
Once the claim is established, exploiting Schauder's fixed point theorem and choosing $k>Cn^{1+\gamma}$, one proves the existence of $u_n,v_n\in\HH(\Omega)\cap L^\infty(\Omega)$ that fulfil \eqref{appr eq1} and \eqref{appr eq2} for any $\varphi,\psi\in\HH(\Omega)$, and such that $u_n>0$ and $v_n\ge0$ \Ae in $\Omega$. The assertion on the strict positivity of $v_n$ follows from the weak Harnack inequality for the linear operator $-\Div(A(x)D \cdot)$, noting that
$$\int\limits_\Omega ADv_n\cdot D\psi\dex\ge0\quad\forall\psi\in\HH(\Omega),\,\psi\ge0,$$
and that $u_n^r\not\equiv0$ (see Theorem $8.18$ in \cite{gilbarg2001elliptic}). To conclude, we have to prove Claim \ref{claim compactness}.
\begin{proof}[Proof of Claim \ref{claim compactness}]
We show that, if $\phi_j\rightharpoonup \phi$ weakly in $\HH(\Omega)$, then, up to subsequences, $S(\phi_j)\to S(\phi)$ strongly in $\HH(\Omega)$. With this aim, take $(\phi_j)$ that converges to $\phi$ weakly in $\HH(\Omega)$ and, without loss of generality, \Ae in $\Omega$. Thanks to the uniform bounds on $(v_{\phi_j})$ and $(u_{\phi_j})$, there exist $z,\xi\in\HH(\Omega)$ such that, up to non-relabelled subsequences, $(v_{\phi_j})$ and $(u_{\phi_j})$ converge, respectively, to $z$ and $\xi$ weakly in $\HH(\Omega)$, strongly in $L^p(\Omega)$ for all $1\le p<2^\star$, and \Ae in $\Omega$. Clearly, $z,\xi\ge0$ \Ae in $\Omega$. Since
\begin{gather*}
\frac{T_k(|\phi_j|)^r}{\big(v_{\phi_j}+\frac{1}{n}\big)^{\theta}}\le n^\theta k^r,\\
\frac{f_n}{\big(u_{\phi_j}+\frac{1}{n}\big)^\gamma}+v_{\phi_j}^{1-\theta} u_{\phi_j}^{r-1}\le n^{1+\gamma}+Ck^{r(1-\theta)}n^{\theta(1-\theta)}n^{(r-1)(1+\gamma)},
\end{gather*}
by the dominated convergence theorem one has
\begin{gather*}
\frac{T_k(|\phi_j|)^r}{\big(v_{\phi_j}+\frac{1}{n}\big)^{\theta}}\to \frac{T_k(|\phi|)^r}{\big(z+\frac{1}{n}\big)^{\theta}},\quad \text{strongly in}\ L^q(\Omega),\\
\frac{f_n}{\big(u_{\phi_j}+\frac{1}{n}\big)^\gamma}-v_{\phi_j}^{1-\theta} u_{\phi_j}^{r-1}\to \frac{f_n}{\big(\xi+\frac{1}{n}\big)^\gamma}-z^{1-\theta} \xi^{r-1}\quad \text{strongly in}\ L^q(\Omega),
\end{gather*}
for any $1\le q<\infty$.
Given that $v_{\phi_j}\rightharpoonup z$ in $\HH(\Omega)$, it follows that
$$\int\limits_\Omega ADz\cdot D\psi\dex=\int\limits_\Omega \frac{T_k(|\phi|)^r}{\big(z+\frac{1}{n}\big)^{\theta}}\psi\dex\quad\forall\psi\in\HH(\Omega),$$
whence, by uniqueness, $z=v_\phi$. Taking $v_{\phi_j}-v_\phi$ as a test function in the weak formulation of the problems solved by $v_{\phi_j}$ and $v_\phi$ yields
\begin{align*}
\alpha\int\limits_\Omega |D(v_{\phi_j}-v_\phi)|^2\dex&\le\int\limits_\Omega AD(v_{\phi_j}-v_\phi)\cdot D(v_{\phi_j}-v_\phi)\dex\\
&=\int\limits_\Omega\bigg(\frac{T_k(|\phi_j|)^r}{\big(v_{\phi_j}+\frac{1}{n}\big)^{\theta}}-\frac{T_k(|\phi_j|)^r}{\big(v_\phi+\frac{1}{n}\big)^{\theta}}\bigg)(v_{\phi_j}-v_\phi)\dex.
\end{align*}
Thanks again to the uniform $L^\infty$--bound on $(v_{\phi_j})$, by the dominated convergence theorem the integral on the r\text{.}h\text{.}s\text{.} vanishes as $j\to\infty$, hence $v_{\phi_j}\to v_{\phi}$ strongly in $\HH(\Omega)$. Then, by the weak convergence of $(u_{\phi_j})$ to $\xi$,
\begin{equation*}
\int\limits_\Omega AD\xi\cdot D\varphi+v_\phi^{1-\theta}\xi^{r-1}\varphi\dex=\int\limits_\Omega \frac{f_n}{\big(\xi+\frac{1}{n}\big)^{\gamma}}\varphi\dex\quad\forall\varphi\in\HH(\Omega),
\end{equation*}
and, by uniqueness, $\xi=u_\phi$. Therefore,
taking $u_{\phi_j}-u_\phi$ as test function in the problems solved by $u_{\phi_j}$ and $u_{\phi}$, and arguing as before, one easily deduces that $u_{\phi_j}\to u_\phi$ strongly in $\HH(\Omega)$, and the proof is concluded. \hfill $\square$
\end{proof}
\end{proof}
We conclude this section proving some regularity estimates on $(u_n)$. Part of our argument is inspired by \cite{BoccardoOrsina2010}; nevertheless, being our problem slightly different, for the reader's convenience we present a full proof.
\begin{lemma}\label{a priori estimates on (un)}
Let $(u_n)$ be the sequence given by Lemma \ref{approximation lemma}. Then:\\
$(i)$ If $m>\frac{d}{2}$, then there exists a constant $C=C(d,|\Omega|,\alpha,m)>0$ such that
$$\|u_n\|_{L^\infty}\le 1+C\|f\|_{L^m}.$$
$(ii)$ If $1\le m<\frac{d}{2}$, then there exists a constant $C=C(d,|\Omega|,\alpha,m,\gamma)>0$ such that
$$\|u_n\|_{L^{m^{\star\star}(1+\gamma)}}\le C\|f\|_{L^m}^{\frac{1}{1+\gamma}}.$$
$(iii)$ If $m=\frac{d}{2}$, then for every $1\le p<\infty$ there exists a constant $C=C(d,|\Omega|,\alpha,\gamma, p)>0$ such that
$$\|u_n\|_{L^p}\le C\|f\|_{L^{\frac{d}{2}}}^{\frac{1}{1+\gamma}}.$$
\end{lemma}
\begin{proof}
$(i)$ Assume $m>\frac{d}{2}$. Since $T_n(f)\le f$, given $k\ge 1$, taking $G_k(u_n)$ as a test function in \eqref{appr eq1} yields
$$\alpha\int\limits_\Omega | D  G_k(u_n)|^2\dex\le\int\limits_\Omega f G_k(u_n)\dex.$$
Thanks to the integrability assumption on $f$, by Theorem \ref{cor stamacchia theorem} there exists a positive constant $C=C(d,|\Omega|,\alpha,m)$ such that
\begin{equation*}
\|u_n\|_{L^\infty}\le 1+C\|f\|_{L^m}.
\end{equation*}

$(ii)$ Assume $\big(\frac{2^\star}{1-\gamma}\big)'\le m<\frac{d}{2}$. Here, we follow the argument in Lemma $5.5$ in \cite{BoccardoOrsina2010}. If $m=\big(\frac{2^\star}{1-\gamma}\big)'$, the conclusion follows from the Sobolev embedding, since in this case
$m^{\star\star}=\frac{2^\star}{1+\gamma}$. If $\big(\frac{2^\star}{1-\gamma}\big)'< m<\frac{d}{2}$, testing \eqref{appr eq1} against $u_n^{\delta}$, $\delta>1$, we get
\begin{align}\label{passage}
\alpha\delta\int\limits_\Omega |Du_n|^2 u_n^{\delta-1}\dex&\le\int\limits_\Omega ADu_n\cdot Du_n \delta u_n^{\delta-1} +v_n^{1-\theta}u_n^{r-1+\delta}\dex\notag\\
&=\int\limits_\Omega \frac{f_n}{\big(u_n+\frac{1}{n}\big)^\gamma}u_n^\delta\dex\le\int\limits_\Omega f_n u_n^{\delta-\gamma}\dex.
\end{align}
Noting that
$$|D(u_n^{\frac{\delta+1}{2}})|^2=\bigg(\frac{\delta+1}{2}\bigg)^2u_n^{\delta-1}|Du_n|^2,$$
\eqref{passage} entails
$$\alpha\delta\frac{4}{(\delta+1)^2}\int\limits_\Omega |D(u_n^{\frac{\delta+1}{2}})|^2\dex\le\int\limits_\Omega f_n u_n^{\delta-\gamma}\dex.$$
Using Sobolev's inequality on the left and H\"older's inequality on the right produces
$$\alpha\delta\frac{4}{(\delta+1)^2\mathcal{S}^2}\bigg(\int\limits_\Omega u_n^{2^\star\frac{\delta+1}{2}}\dex\bigg)^{\frac{2}{2^\star}}\le \|f\|_{L^m} \bigg(\int\limits_\Omega u_n^{(\delta-\gamma)m'}\dex\bigg)^{\frac{1}{m'}},$$
where $\mathcal{S}$ is the best constant in the Sobolev embedding. Now, we choose $\delta>1$ in such a way that
\begin{equation}\label{choice of delta in ii}
2^\star\frac{\delta+1}{2}=(\delta-\gamma)m',
\end{equation}
which is possible because $m'(1-\gamma)<2^\star$. This implies $2^\star\frac{\delta+1}{2}=m^{\star\star}(1+\gamma)$, and the conclusion follows taking 
\begin{equation}\label{constant case ii}
C=\bigg(\frac{(\delta+1)^2}{4\alpha\delta \mathcal{S}^2}\bigg)^\frac{1}{1+\gamma}.
\end{equation}
Now, assume $1\le m<\big(\frac{2^\star}{1-\gamma}\big)'$. Here, we follow Theorem $5.6$ in \cite{BoccardoOrsina2010}. Taking $(u_n+\varepsilon)^\delta$ as a test function in \eqref{appr eq1}, $0<\varepsilon<\frac{1}{n}$ and $\gamma<\delta<1$, and arguing as in the case $\big(\frac{2^\star}{1-\gamma}\big)'\le m<\frac{d}{2}$, one deduces that
$$\alpha\delta\frac{4}{(\delta+1)^2\mathcal{S}^2}\bigg(\int\limits_\Omega (u_n+\varepsilon)^{2^\star\frac{\delta+1}{2}}\dex\bigg)^{\frac{2}{2^\star}}\le \|f\|_{L^m} \bigg(\int\limits_\Omega (u_n+\varepsilon)^{(\delta-\gamma)m'}\dex\bigg)^{\frac{1}{m'}}.$$
Letting $\varepsilon$ tend to $0$, we obtain again
$$\alpha\delta\frac{4}{(\delta+1)^2\mathcal{S}^2}\bigg(\int\limits_\Omega u_n^{2^\star\frac{\delta+1}{2}}\dex\bigg)^{\frac{2}{2^\star}}\le \|f\|_{L^m} \bigg(\int\limits_\Omega u_n^{(\delta-\gamma)m'}\dex\bigg)^{\frac{1}{m'}}.$$
If we take $\gamma<\delta<1$ such that $2^\star\frac{\delta+1}{2}=(\delta-\gamma)m'$, the conclusion follows noting that, with this choice, $2^\star\frac{\delta+1}{2}=m^{\star\star}(1+\gamma)$.

$(iii)$ Let $m=\frac{d}{2}$. Let $1\le p<\infty$ and choose $\varepsilon>0$ such that
\begin{equation}\label{choice of epsilon in iii}
\frac{d}{2}-\varepsilon>\bigg(\frac{2^\star}{1-\gamma}\bigg)',\quad \bigg(\frac{d}{2}-\varepsilon\bigg)^{\star\star}(1+\gamma)\ge p.
\end{equation}
Since $f\in L^{\frac{d}{2}-\varepsilon}(\Omega)$, by the previous case it follows that $(u_n)$ is bounded in $L^{(\frac{d}{2}-\varepsilon)^{\star\star}(1+\gamma)}(\Omega)$. In particular, set $q_\varepsilon=(\frac{d}{2}-\varepsilon)^{\star\star}(1+\gamma)$, we deduce
$$\|u_n\|_{L^p}\le |\Omega|^{\frac{1}{p}-\frac{1}{q_\varepsilon}}\|u_n\|_{L^{q_\varepsilon}}\le c|\Omega|^{\frac{1}{p}-\frac{1}{q_\varepsilon}} \|f\|_{L^{\frac{d}{2}-\varepsilon}}^{\frac{1}{1+\gamma}}\le c|\Omega|^{\frac{1}{p}-\frac{1}{q_\varepsilon}}|\Omega|^{\frac{1}{1+\gamma}(\frac{2}{d-2\varepsilon}-\frac{2}{d})} \|f\|_{L^{\frac{d}{2}}}^{\frac{1}{1+\gamma}}.$$
where $c=c(d,|\Omega|,\alpha,\varepsilon,\gamma)>0$ comes from \eqref{constant case ii}. The conclusion follows taking
\begin{equation}\label{choice of C in iii}
C=c|\Omega|^{\frac{1}{p}-\frac{1}{q_\varepsilon}}|\Omega|^{\frac{1}{1+\gamma}(\frac{2}{d-2\varepsilon}-\frac{2}{d})}
\end{equation}
and noting that the $\varepsilon$ chosen in \eqref{choice of epsilon in iii} depends on $p$, $d$, and $\gamma$.
\end{proof}
\begin{rmk}
The constant $C$ given by the proof of $(iii)$ in Lemma \ref{a priori estimates on (un)} blows up as $p\to\infty$; in fact, by \eqref{choice of C in iii}, \eqref{choice of epsilon in iii}, \eqref{constant case ii}, and \eqref{choice of delta in ii}, it is easy to show that $C$ behaves as $p^{\frac{1}{1+\gamma}}$ as $p\to\infty$. Hence, we cannot deduce that $(u_n)$ is bounded in $L^\infty(\Omega)$.
\end{rmk}
\section{Proof of Theorem \ref{existence m>d/2}}\label{regular data}
Let $(u_n)$ and $(v_n)$ be the sequences given by Lemma \ref{approximation lemma}. We divide the proof into steps.
\subsection{Energy boundedness} First, we show that $(u_n)$ and $(v_n)$ are bounded in $\HH(\Omega)$. By Lemma \ref{a priori estimates on (un)},
\begin{equation}\label{a priori pm on (un)}
\|u_n\|_{L^{p_m}}\le C
\end{equation}
where
$$p_m=\begin{cases}
\infty &\mbox{if}\ m>\frac{d}{2}\\
\text{any}\ p\in[1,\infty) &\mbox{if}\ m=\frac{d}{2}
\end{cases}$$
and $C=C(d,|\Omega|,\alpha,m,\|f\|_{L^m},p)>0$. Then, choosing $p_m=\frac{m'}{1-\gamma}$ when $m=\frac{d}{2}$,
\begin{gather*}
\int\limits_\Omega fu_n^{1-\gamma}\dex\le \|u_n\|_{L^{m'}}\|f\|_{L^m}\le C\|f\|_{L^m}.
\end{gather*}
On the other hand, the choices of $p_m>\frac{d}{2}$ if $m=\frac{d}{2}$, and of $G_k(v_n)$, $k\ge1$, as a test function in \eqref{appr eq2}, lead to
$$\alpha\int\limits_\Omega | D  G_k(v_n)|^2\dex\le\int\limits_\Omega u_n^r G_k(v_n)\dex.$$
Then, by Theorem \ref{cor stamacchia theorem} there exists $C=C(d,|\Omega|,\alpha,p_m)>0$ such that
\begin{equation}\label{a priori Linfty on (vn)}
\|v_n\|_{L^\infty}\le 1+C\|u_n^r\|_{L^{p_m}}\le C_1.
\end{equation}
Therefore, taking $u_n$ and $v_n$ as test functions in \eqref{appr eq1} and \eqref{appr eq2}, respectively, produces the following a priori energy estimates:
\begin{gather}
\alpha\|u_n\|_{\HH}^2\le \int\limits_\Omega \frac{f_n}{\big(u_n+\frac{1}{n}\big)^\gamma}u_n\dex\le\int\limits_\Omega f u_n^{1-\gamma}\dex\le C\|f\|_{L^m},\label{a priori energy 1 on (un)}\\
\alpha\|v_n\|_{\HH}^2\le\int\limits_\Omega \frac{u_n^r}{\big(v_n+\frac{1}{n}\big)^{\theta}}v_n\dex\le\int\limits_\Omega u_n^r v_n^{{1-\theta}}\dex\le C_1^{1-\theta} \|u_n\|_{L^r}^r\le C_2.\label{a priori energy 1 on (vn)}
\end{gather}
\subsection[Construction of \mathtitle{$u$} and \mathtitle{$v$}]{Construction of \mathtitle{${\boldsymbol{u}}$} and \mathtitle{$\boldsymbol{v}$}} By \eqref{a priori energy 1 on (un)} and \eqref{a priori energy 1 on (vn)}, there exist $u,v\in\HH(\Omega)$ such that, up to subsequences, $(u_n)$ and $(v_n)$ converge, respectively, to $u$ and $v$ weakly in $\HH(\Omega)$, strongly in $L^q(\Omega)$ for all $1\le q<2^\star$, and \text{a.e.} in $\Omega$. Moreover, by \eqref{a priori pm on (un)} and \eqref{a priori Linfty on (vn)}, $u\in L^p(\Omega)$ for all $1\le p<\infty$, and $v\in L^\infty(\Omega)$. Since $u_n,v_n>0$ \Ae in $\Omega$, it follows that $u,v\ge0$ \Ae in $\Omega$.
\subsection[\mathtitle{$u$} and \mathtitle{$v$} are positive]{\mathtitle{$\boldsymbol{u}$} and \mathtitle{$\boldsymbol{v}$} are positive} We show that $u,v>0$ \Ae in $\Omega$. By \eqref{a priori pm on (un)} and \eqref{a priori Linfty on (vn)}, the sequence $(v_n^{1-\theta}u_n^{r-2})$ is bounded in $L^\infty(\Omega)$, so $v^{1-\theta}u^{r-2}\in L^\infty(\Omega)$.
Let us define the function $g\colon\Omega\to[0,+\infty]$ by setting
$$g\Def\begin{cases}
\displaystyle 0 &\mbox{in}\ \{f=0\}\\
\displaystyle\frac{f}{u^\gamma} &\mbox{in}\ \{f>0\}\cap\{u\neq0\}\\
+\infty &\mbox{in}\ \{f>0\}\cap\{u=0\}.
\end{cases}$$
Clearly, $f_n (u_n+1/n)^{-\gamma}\to g$ \Ae in $\Omega$. Then, for any $0\le \varphi\in\HH(\Omega)$,
\begin{align*}
\int\limits_\Omega A D u\cdot D \varphi+v^{1-\theta} u^{r-1}\varphi\dex&=\lim_{n\to\infty}\int\limits_\Omega A D  u_{n}\cdot D \varphi\dex+\lim_{n\to\infty}\int\limits_\Omega v_{n}^{1-\theta} u_{n}^{r-1}\varphi\dex\\
&=\liminf_{n\to\infty}\int\limits_\Omega \frac{f_n}{\big(u_{n}+\frac{1}{n}\big)^\gamma}\varphi\dex\ge\int\limits_\Omega g\varphi\dex,
\end{align*}
where the first equality follows from Lebesgue's dominated convergence theorem and the weak convergence of $(u_n)$ to $u$, whereas the inequality in the second line relies upon Fatou's lemma. It follows that $g\varphi\in L^1(\Omega)$ for all $0\le\varphi\in \HH(\Omega)$, whence $g$ is finite \Ae in $\Omega$, that is, $u>0$ \Ae in $\{f>0\}$. Then, $u$ does not vanish identically and satisfies
$$\int\limits_\Omega A D u\cdot D \varphi+v^{1-\theta} u^{r-2}u\varphi\dex\ge0\quad\forall\varphi\in\HH(\Omega),\,\varphi\ge0,$$
Since $v$ is bounded and $u$ belongs to every $L^p(\Omega)$, $1\le p<\infty$, it follows that $v^{1-\theta}u^{r-2}\in L^s(\Omega)$ for all $s>\frac{d}{2}$. In particular, by the weak Harnack inequality (see Theorem $7.1.2$ in \cite{pucci-serrin}), $u>0$ \Ae in $\Omega$, and $g=fu^{-\gamma}$. Therefore,
\begin{equation}\label{u super-solution m>d/2}
\int\limits_\Omega A D u\cdot D \varphi+ v^{1-\theta} u^{r-1}\varphi\dex\ge \int\limits_\Omega \frac{f}{u^\gamma}\varphi\dex\quad\forall\varphi\in\HH(\Omega),\,\varphi\ge0.
\end{equation}
Concerning the positivity of $v$, one defines $h\colon\Omega\to[0,+\infty]$ as
$$h\Def\begin{cases}
\displaystyle\frac{u^r}{v^{\theta}} &\mbox{in}\ \{v\neq0\}\\
+\infty &\mbox{otherwise}
\end{cases}$$
and, using \eqref{appr eq2}, deduces
$$\int\limits_\Omega ADv\cdot D\psi\dex\ge\int\limits_\Omega h\psi\dex\ge0\quad\forall\psi\in\HH(\Omega),\,\psi\ge0.$$
Then, invoking again the weak Harnack inequality, $v>0$ \Ae in $\Omega$ and
\begin{equation}\label{v super-solution m>d/2}
\int\limits_\Omega ADv\cdot D\psi\dex\ge\int\limits_\Omega \frac{u^r}{v^{\theta}}\psi\dex\quad\forall\psi\in\HH(\Omega),\,\psi\ge0.
\end{equation}
Note that, by \eqref{u super-solution m>d/2} and \eqref{v super-solution m>d/2}, $u$ and $v$ satisfy \eqref{integrability of reactions weak solutions}.
\subsection[\mathtitle{$(u,v)$} is a weak solution]{\mathtitle{$\boldsymbol{(u,v)}$} is a weak solution} By \eqref{u super-solution m>d/2} and \eqref{v super-solution m>d/2}, $u$ and $v$ are super-solutions. We have to show that the converse inequalities hold.

First, we deal with the first equation. Given $\varepsilon>0$ and $0\le\varphi\in \HH(\Omega)\cap L^\infty(\Omega)$, since $u>0$ \Ae in $\Omega$, the function $\frac{u_n}{u+\varepsilon}\varphi$ is a well defined test function for \eqref{appr eq1}. Then, we have
\begin{multline*}
\int\limits_\Omega A D  u_{n}\cdot D  u_{n}\frac{\varphi}{u+\varepsilon}- A Du_{n}\cdot D u\frac{u_{n}}{(u+\varepsilon)^2}\varphi\dex\\
+\int\limits_\Omega A D u_{n}\cdot D \varphi\frac{u_{n}}{u+\varepsilon}+v_{n}^{1-\theta} \frac{u_n^r}{u+\varepsilon}\varphi\dex\le\int\limits_\Omega f_n \frac{u_{n}^{1-\gamma}}{u+\varepsilon}\varphi\dex.
\end{multline*}
Thanks to the weak and the \text{a.e.} convergence of $(u_{n})$ to $u$, the uniform $L^\infty$--bound on $(u_{n})$, the dominated convergence theorem, and the weak sequential lower semi--continuity of the quadratic form
$$\HH(\Omega)\ni w\longmapsto\int\limits_\Omega A D  w\cdot D  w\dex,$$
letting $n\to\infty$ yields
\begin{multline}\label{final step subsol thm data bounded 1st eq}
\int\limits_\Omega A D u\cdot D u\frac{\varphi}{u+\varepsilon}- A D  u\cdot D u\frac{u\varphi}{(u+\varepsilon)^2}\dex\\
+\int\limits_\Omega A D u\cdot D \varphi\frac{u}{u+\varepsilon}+ v^{1-\theta} \frac{u^r}{u+\varepsilon}\varphi\dex\le\int\limits_\Omega f \frac{u^{1-\gamma}}{u+\varepsilon}\varphi\dex.
\end{multline}
Note that $f\frac{u^{1-\gamma}}{u+\varepsilon}\varphi\le f u^{-\gamma}\varphi\in L^1(\Omega)$ for all $\varepsilon>0$. Then, disregarding the first integral in \eqref{final step subsol thm data bounded 1st eq} (which is positive), and applying Fatou's lemma on the l\text{.}h\text{.}s\text{.} and the dominated convergence theorem on the r\text{.}h\text{.}s\text{.} (w\text{.}r\text{.}t\text{.} $\varepsilon\to0$), we obtain
\begin{equation}\label{u sub-solution m>d/2}
\int\limits_\Omega A D  u\cdot D \varphi+v^{1-\theta} u^{r-1}\varphi\dex\le\int\limits_\Omega\frac{f}{u^\gamma}\varphi\dex\quad\forall\varphi\in\HH(\Omega)\cap L^\infty(\Omega),\, \varphi\ge0.
\end{equation}
Combining \eqref{u super-solution m>d/2} and \eqref{u sub-solution m>d/2} it follows
\begin{equation}\label{non lo so}
\int\limits_\Omega A D u\cdot D \varphi+ v^{1-\theta} u^{r-1}\varphi\dex=\int\limits_\Omega\frac{f}{u^\gamma}\varphi\dex\quad\forall\varphi\in \HH(\Omega)\cap L^\infty(\Omega).
\end{equation}
Now, we have to show that \eqref{non lo so} holds for unbounded test functions. With this aim, let $0\le\varphi\in \HH(\Omega)$ and $k\in\N$. By \eqref{non lo so}, it follows that
$$\int\limits_\Omega A D u\cdot D T_k(\varphi)+v^{1-\theta} u^{r-1}T_k(\varphi)\dex=\int\limits_\Omega\frac{f}{u^\gamma}T_k(\varphi)\dex.$$
Then, by the dominated and the monotone convergence theorems, letting $k\to\infty$ entails
$$\int\limits_\Omega A D u\cdot D \varphi+v^{1-\theta} u^{r-1}\varphi\dex=\int\limits_\Omega\frac{f}{u^\gamma}\varphi\dex\quad\forall\varphi\in \HH(\Omega),\,\varphi\ge0,$$
which, by linearity with respect to $\varphi$, entails the desired conclusion.

To pass to the limit in the second equation, given $\varepsilon>0$ and $0\le\psi\in\HH(\Omega)\cap L^\infty(\Omega)$, one takes $\frac{v_n}{v+\varepsilon}\psi$ as a test function in \eqref{appr eq2}, and argues as above to get
\begin{equation}\label{final step eq 2 regular data}
\int\limits_\Omega A D v\cdot D \psi\dex=\int\limits_\Omega\frac{u^r}{v^{\theta}}\psi\dex\quad\forall\psi\in \HH(\Omega)\cap L^\infty(\Omega).
\end{equation}
Then, we conclude taking $\psi=T_k(\varphi)$ in \eqref{final step eq 2 regular data}, with $k\in\N$ and $0\le\varphi\in\HH(\Omega)$, letting $k\to\infty$, and, finally, exploiting the linearity w\text{.}r\text{.}t\text{.} the test function.
\section{Data in the dual space}\label{data dual space}
Let $(u_n)$ and $(v_n)$ be the sequences given by Lemma \ref{approximation lemma}. We first prove two auxiliary results.
\begin{lemma}
$(i)$ For any $\psi\in \HH(\Omega)$, $\psi\ge0$,
\begin{equation}\label{inequality for second equation}
\int\limits_{\{v_n\le1\}} \frac{u_n^r}{\big(v_n+\frac{1}{n}\big)^{\theta}}\psi\dex \le\int\limits_\Omega A D  T_1(v_n)\cdot D \psi\dex.
\end{equation}

$(ii)$ For every $h>0$,
\begin{equation}\label{bound for integral on superlevels of un}
\int\limits_{\{u_n\ge h\}}\frac{u_n^r}{\big(v_n+\frac{1}{n}\big)^{\theta}}\dex\le\frac{\beta}{2h}\Big(\|u_n\|_{\HH}^2+\|v_n\|_{\HH}^2\Big).
\end{equation}
\end{lemma}
\begin{proof}
$(i)$ Let $\varepsilon>0$ and $\psi\in\HH(\Omega)$, $\psi\ge0$, and set $F_\varepsilon (v_n)= 1-\frac{1}{\varepsilon}T_\varepsilon(G_1(v_n)).$ Taking $F_\varepsilon (v_n)\psi$ as a test function in the weak formulation of the second equation yields
$$-\frac{1}{\varepsilon}\int\limits_\Omega A D  v_n\cdot D  v_n\chi_{\{1\le v_n\le 1+\varepsilon\}}\psi \dex+\int\limits_\Omega A D  v_n\cdot D  \psi F_\varepsilon (v_n)\dex = \int\limits_\Omega \frac{u_n^r}{\big(v_n+\frac{1}{n}\big)^{\theta}}F_\varepsilon (v_n)\psi\dex.$$
Dropping the first integral (which is negative), we obtain
$$\int\limits_\Omega \frac{u_n^r}{\big(v_n+\frac{1}{n}\big)^{\theta}}F_\varepsilon (v_n)\psi\dex\le\int\limits_\Omega A D  v_n\cdot D  \psi F_\varepsilon (v_n)\dex$$
and letting $\varepsilon\to0$ \eqref{inequality for second equation} follows. 

$(ii)$ Note that $|A D v_n\cdot D  u_n|\le\beta| D  v_n| \,| D  u_n|$.
Then
\begin{align*}
\int\limits_{\{u_n\ge h\}}\frac{u_n^r}{\big(v_n+\frac{1}{n}\big)^{\theta}}\dex&\le\frac{1}{h}\int\limits_{\{u_n\ge h\}}\frac{u_n^{r+1}}{\big(v_n+\frac{1}{n}\big)^{\theta}}\dex\le\frac{1}{h}\int\limits_{\Omega}A D  v_n\cdot D  u_n\dex\\
&\le\frac{1}{h}\int\limits_{\Omega}\beta| D  v_n| \,| D  u_n|\dex\le\frac{\beta}{2h}\Big(\|u_n\|_{\HH}^2+\|v_n\|_{\HH}^2\Big),
\end{align*}
where the last inequality follows from Young's inequality.
\end{proof}
\begin{lemma}\label{a priori estimates on (vn)}
Let $r=2$ and $\big(\frac{2^\star}{1-\gamma}\big)'\le m<\frac{d}{2}$. Then:\\
$(i)$ If $m>\frac{d}{3+\gamma}$, then there exists a constant $C=C(d,|\Omega|,\alpha,m,\gamma)>0$ such that
$$\|v_n\|_{L^\infty}\le 1+C\|f\|_{L^m}^{\frac{2}{1+\gamma}}.$$
$(ii)$ If $\big(\frac{2^\star}{1-\gamma}\big)'\le m<\frac{d}{3+\gamma}$, then there exists a constant $C=C(d,|\Omega|,\alpha,m,\gamma,\theta)>0$ such that
$$\|v_n\|_{L^{s^{\star\star}(1+\theta)}}\le C\|f\|_{L^m}^{\frac{2}{(1+\theta)(1+\gamma)}},$$
where $s=\frac{m^{\star\star}(1+\gamma)}{2}$.\\
$(iii)$ If $m=\frac{d}{3+\gamma}$, then for every $1\le p<\infty$ there exists a constant $C=C(d,|\Omega|,\alpha,\gamma, \theta,p)>0$ such that
$$\|v_n\|_{L^p}\le C\|f\|_{L^{m}}^{\frac{2}{(1+\theta)(1+\gamma)}}.$$
\end{lemma}
\begin{proof}
By Lemma \ref{a priori estimates on (un)}, there exists a constant $C=C(d,|\Omega|,\alpha,m,\gamma)>0$ such that
\begin{equation}\label{estimate (un) lemma for (vn)}
\|u_n\|_{L^{m^{\star\star}(1+\gamma)}}\le C\|f\|_{L^m}^{\frac{1}{1+\gamma}}.
\end{equation}
Since $r=2$, \eqref{appr eq2} becomes
$$\int\limits_\Omega ADv_n\cdot D\varphi\dex=\int\limits_\Omega \frac{u_n^2}{\big(v_n+\frac{1}{n}\big)^\theta}\varphi\dex\quad\forall\varphi\in\HH(\Omega).$$
Note that the argument that proves Lemma \ref{a priori estimates on (un)} applies to any sequence $(w_n)\subset\HH(\Omega)\cap L^\infty(\Omega)$ such that $w_n$ is a non-negative weak solution to
$$-\Div(A(x)Dw_n)\le\frac{g_n}{(w_n+\frac{1}{n}\big)^\theta},$$
where $0\le g_n\in L^m(\Omega)$ is a sequence of functions whose $L^m(\Omega)$--norms are uniformly bounded. Note also that
$$\frac{m^{\star\star}(1+\gamma)}{2}\ge\frac{\Big(\big(\frac{2^\star}{1-\gamma}\big)'\Big)^{\star\star}(1+\gamma)}{2}\ge 1.$$
Since
$$\frac{m^{\star\star}(1+\gamma)}{2}\ge\frac{d}{2}\quad\text{if and only if}\quad m\ge\frac{d}{3+\gamma},$$
the conclusion follows applying the argument in Lemma \ref{a priori estimates on (un)} to $(v_n)$ and the second equation, with $m$ replaced by $\frac{m^{\star\star}(1+\gamma)}{2}$, and exploiting \eqref{estimate (un) lemma for (vn)} to estimate $\|u_n^2\|_{L^{m^{\star\star}(1+\gamma)}}^{\frac{1}{1+\gamma}}$.
\end{proof}
\begin{proof}[\textbf{Proof of Theorem \ref{existence data in dual space}}]
\emph{Step $1$: energy boundedness.} As for Theorem \ref{existence m>d/2}, we show that $(u_n)$ and $(v_n)$ have bounded energy. Taking $u_n$ and $v_n$ as test functions in \eqref{appr eq1} and \eqref{appr eq2}, respectively, we get
\begin{gather*}
\int\limits_\Omega A D  u_n\cdot D  u_n+v_n^{1-\theta} u_n^{r}\dex=\int\limits_\Omega \frac{f_n}{\big(u_n+\frac{1}{n}\big)^\gamma} u_n\dex,\ \ \int\limits_\Omega A D  v_n\cdot D  v_n\dex=\int\limits_\Omega  \frac{u_n^r}{\big(v_n+\frac{1}{n}\big)^{\theta}}v_n\dex.
\end{gather*}
The above formulae and H\"older's inequality yield
\begin{multline}\label{holder dual data}
\alpha\int\limits_\Omega |D  u_n|^2+| D  v_n|^2\dex\le \int\limits_\Omega A D  u_n\cdot D  u_n\dex+\int\limits_\Omega A D  v_n\cdot D  v_n\dex \le\\
\le \int\limits_\Omega f u_n^{1-\gamma}\dex\le C\|f\|_{L^m}\bigg(\int\limits_\Omega u_n^{m'(1-\gamma)}\dex\bigg)^{1/m'}.
\end{multline}
Since $m'(1-\gamma)<2^\star$ and, by Lemma \ref{a priori estimates on (un)}, $\|u_n\|_{L^{2^\star}}\le C\|f\|_{L^m}^{\frac{1}{1+\gamma}}$, applying again H\"older's inequality on the right of \eqref{holder dual data} we obtain
$$\|u_n\|_{\HH}^2+\|v_n\|_{\HH}^2\le C\|f\|_{L^m} \bigg(\int\limits_\Omega u_n^{2^\star}\dex\bigg)^{(1-\gamma)/2^\star}\le C\|f\|_{L^m}^{\frac{2}{1+\gamma}}.$$
Therefore, there exist $u,v\in\HH(\Omega)$ such that, up to subsequences, $(u_n)$ and $(v_n)$ converge, respectively, to $u$ and $v$ weakly in $\HH(\Omega)$, strongly in $L^p(\Omega)$ for all $1\le p<2^\star$, and \text{a.e.} in $\Omega$. Moreover, the higher integrability of $u$ follows directly from Lemma \ref{a priori estimates on (un)}.

\emph{Step $2$.} We show that
\begin{equation}\label{conv L1 first equation data in dual space}
v_n^{1-\theta}u_n^{r-1}\to v^{1-\theta} u^{r-1}\quad\text{strongly in}\ L^1(\Omega).
\end{equation}
To do this, note at first that, for any $h>0$,
\begin{align*}
\int\limits_{\{u_n\ge h\}}v_n^{1-\theta}u_n^{r}\dex\le \int\limits_\Omega f u_n^{1-\gamma}\dex\le |\Omega|^{\frac{1}{m'}-\frac{1-\gamma}{2^\star}}\|f\|_{L^m} \bigg(\int\limits_\Omega u_n^{2^\star}\dex\bigg)^{(1-\gamma)/2^\star}.
\end{align*}
By the integrability assumption on $f$, and the boundedness of $(u_n)$ in $L^{2^\star}(\Omega)$ proved above, there exists a positive constant $C$ independent from $h$ such that
$$\int\limits_{\{u_n\ge h\}}v_n^{1-\theta}u_n^{r}\dex\le C\quad\forall n\in\N,\, h>0.$$
Then, for $E\subset\Omega$ measurable we have
\begin{align*}
\int\limits_{E} v_n^{1-\theta}u_n^{r-1}\dex\le h^{r-1}\int\limits_{E\cap\{u_n\le h\}} v_n^{1-\theta}\dex+\frac{1}{h}\int\limits_{E\cap\{u_n\ge h\}}v_n^{1-\theta}u_n^{r}\dex\le h^{r-1}\int\limits_{E} v_n^{1-\theta}\dex+\frac{C}{h}.
\end{align*}
Now, given $\varepsilon>0$, there exists $h_\varepsilon>0$ such that $C/{h_\varepsilon}<\varepsilon/{2}.$
Moreover, by Young's inequality
$$\int\limits_E v_n^{1-\theta}\dex\le \frac{{1-\theta}}{2}\int\limits_E v_n^2\dex+\frac{1+\theta}{2}|E|.$$
Recalling that $v_n\to v$ in $L^2(\Omega)$, by Vitali's theorem we can find $\delta_\varepsilon>0$ such that, if $|E|<\delta_\varepsilon$, one has
$$h_\varepsilon^{r-1}\int\limits_{E} v_n^{1-\theta}\dex<\frac{\varepsilon}{2}.$$
It follows that
$$\int\limits_{E} v_n^{1-\theta}u_n^{r-1}\dex<\varepsilon$$
for all measurable set $E\subset\Omega$ such that $|E|<\delta_\varepsilon$, whence $(u_n^{r-1}v_n^{{1-\theta}})$ is equi-integrable and \eqref{conv L1 first equation data in dual space} follows from Vitali's theorem.

\emph{Step $3$: $u$ and $v$ are super-solutions.} By the previous step, for any $\varphi\in \HH(\Omega)\cap L^\infty(\Omega)$, $u_n^{r-1}v_n^{1-\theta}\varphi\to u^{r-1}v^{1-\theta}\varphi$ strongly in $L^1(\Omega)$. This implies that
$$\int\limits_\Omega A D u_n\cdot D \varphi+ v_n^{1-\theta} u_n^{r-1}\varphi\dex\to \int\limits_\Omega A D u\cdot D \varphi+ v^{1-\theta} u^{r-1}\varphi\dex$$
for all $\varphi\in \HH(\Omega)\cap L^\infty(\Omega)$.
Then, letting
$$g(x,u)\Def\begin{cases}
\displaystyle 0 &\mbox{in}\ \{f=0\},\\
\displaystyle\frac{f(x)}{u^\gamma} &\mbox{in}\ \{f>0\}\cap\{u\neq0\},\\
+\infty &\mbox{in}\ \{f>0\}\cap\{u=0\},
\end{cases}$$
choosing $0\le\varphi\in\HH(\Omega)\cap L^\infty(\Omega)$ in \eqref{appr eq1} and letting $n\to\infty$ yield
\begin{equation}\label{u super-solution data dual space}
\int\limits_\Omega A D u\cdot D \varphi+ v^{1-\theta} u^{r-1}\varphi\dex\ge \int\limits_\Omega g(\cdot,u)\varphi\dex\quad\forall\varphi\in\HH(\Omega)\cap L^\infty(\Omega),\,\varphi\ge0.
\end{equation}
As a consequence, $g(\cdot,u(\cdot))$ is finite \Ae in $\Omega$, and $u>0$ \Ae in $\{f>0\}$. Concerning the second equation, the same argument of the second part of Step $3$ in Theorem \ref{existence m>d/2} applies, and yields $v>0$ \Ae in $\Omega$ and
\begin{equation}\label{v super-solution data dual space}
\int\limits_\Omega ADv\cdot D\psi\dex\ge\int\limits_\Omega \frac{u^r}{v^\theta}\psi\dex\quad\forall\psi\in\HH(\Omega),\,\psi\ge0.
\end{equation}
As a consequence, \eqref{integrability of reactions finite-energy solutions} holds and $u$ and $v$ are super-solutions.

\emph{Step $4$.} Adapting the argument of Step $2$, we now show that, for any $\varepsilon>0$ and $\psi\in \HH(\Omega)\cap L^\infty(\Omega)$,
\begin{equation}\label{conv L^1 second eq data dual space}
    \frac{u_n^r}{\big(v_n+\frac{1}{n}\big)^{\theta}}\frac{T_1(v_n)}{T_1(v)+\varepsilon}\psi\to \frac{u^r}{v^{\theta}}\frac{T_1(v)}{T_1(v)+\varepsilon}\psi\quad\text{strongly in}\ L^1(\Omega).
\end{equation}
Let $E\subset\Omega$ be a measurable set and $h>0$. We have
\begin{gather*}
\int\limits_{E\cap\{u_n\le h\}} \frac{u_n^r}{\big(v_n+\frac{1}{n}\big)^{\theta}}\frac{T_1(v_n)}{T_1(v)+\varepsilon}\psi\dex\le C \|\psi\|_{L^\infty}\varepsilon^{-1} h^r\int\limits_{E} v_n^{1-\theta}\dex\\
\int\limits_{E\cap\{u_n\ge h\}} \frac{u_n^r}{\big(v_n+\frac{1}{n}\big)^{\theta}}\frac{T_1(v_n)}{T_1(v)+\varepsilon}\psi\dex\le C\|\psi\|_{L^\infty}\varepsilon^{-1}\frac{1}{h} \int\limits_{\{u_n\ge h\}} \frac{u_n^{r+1}}{\big(v_n+\frac{1}{n}\big)^{\theta}}\dex.
\end{gather*}
Since $(u_n)$ and $(v_n)$ are bounded in $\HH(\Omega)$, by \eqref{bound for integral on superlevels of un} we get
$$\int\limits_E \frac{u_n^r}{\big(v_n+\frac{1}{n}\big)^{\theta}}\frac{T_1(v_n)}{T_1(v)+\varepsilon}\psi\dex\le C\|\psi\|_{L^\infty}\varepsilon^{-1} h^r\int\limits_{E} v_n^{1-\theta}\dex +C\|\psi\|_{L^\infty}\varepsilon^{-1}\frac{1}{h}$$
for all $h>0$. Then, for any $\sigma>0$, choosing $h_\sigma$ big enough and exploiting Young's inequality as done in Step $1$, we prove that there exists $\delta_\sigma>0$ such that
$$\int\limits_E \frac{u_n^r}{\big(v_n+\frac{1}{n}\big)^{\theta}}\frac{T_1(v_n)}{T_1(v)+\varepsilon}\psi\dex<\sigma$$
whenever $|E|<\delta_\sigma$, and \eqref{conv L^1 second eq data dual space} follows again from Vitali's theorem.

\emph{Step $5$.} We claim that
\begin{equation}\label{ae convergence by Boccardo Murat}
 D  v_n\to D  v\quad \ \text{\Ae in}\ \Omega.
\end{equation}
Let $M$ be a positive constant such that $\| D  v_n\|_{L^2}\le M$ for all $n\in\N$. For every $\phi\in H^1_0(\Omega)$ one has
$$\int\limits_\Omega\frac{u_n^r}{\big(v_n+\frac{1}{n}\big)^{\theta}}\phi\dex=\int\limits_\Omega A D  v_n\cdot D \phi\dex\le \beta \| D  v_n\|_{L^2}\| D \phi\|_{L^2}\le \beta M \|\phi\|_{\HH},$$
whence $u_n^r(v_n+1/n)^{-\theta}\in H^{-1}(\Omega)$. Given $K\subset\Omega$ compact, as it is well known there exist $c\ge1$, $U\ssubset\Omega$ open and $\eta\in C^\infty_{\mathrm{c}}(\Omega)$ such that
$$K\subset U,\quad 0\le\eta\le1,\quad \eta=1\ \text{in}\ K,\quad \eta=0\ \text{in}\ \Omega\setminus \overline{U},\quad \|D\eta\|_{L^\infty(\Omega)}\le \frac{c}{\delta},$$
where $\delta=\mathrm{dist}(K,\partial\Omega)$. Then
$$\int\limits_K \frac{u_n^r}{\big(v_n+\frac{1}{n}\big)^{\theta}}\dex\le \int\limits_\Omega\frac{u_n^r}{\big(v_n+\frac{1}{n}\big)^{\theta}}\eta\dex\le\beta \int\limits_{U\setminus K} |Dv_n||D\eta|\dex\le\beta M\frac{c}{\delta}|U\setminus K|^{1/2}=\vcentcolon C_K,$$
hence, for any $\phi\in C_{\mathrm{c}}(\Omega)$ with $\mathrm{supp}(\phi)\subset K$,
$$\int\limits_\Omega\frac{u_n^r}{\big(v_n+\frac{1}{n}\big)^{\theta}}\phi\dex=\int\limits_{\mathrm{supp}(\phi)}\frac{u_n^r}{\big(v_n+\frac{1}{n}\big)^{\theta}}\phi\dex\le C_K\|\phi\|_{L^\infty(K)}.$$
It follows that $(u_n^r (v_n+1/n)^{-\theta})$ is bounded in the space of Radon measures on $\Omega$. Therefore, \eqref{ae convergence by Boccardo Murat} follows from Theorem \ref{Boccardo-Murat}.

\emph{Step $6$.} We claim that
\begin{equation}\label{ae convergence by Boccardo Murat 2}
 D  u_n\to D  u\quad \ \text{\Ae in}\ \Omega.
\end{equation}
By Step $1$, there exists $0\le h\in L^1(\Omega)$ such that, choosing a subsequence if necessary, $|v_n^{1-\theta}u_n^{r-1}|\le h$ \Ae in $\Omega$. Fix $K\subset\Omega$ compact. Taking $\eta$ as in Step $4$, we have
$$\int\limits_K\frac{f_n}{\big(u_n+\frac{1}{n}\big)^{\gamma}}\dex\le\int\limits_\Omega ADu_n\cdot D\eta+v_n^{1-\theta}u_n^{r-1}\eta\dex\le c_{K} |U\setminus K|^{1/2}+\int\limits_K h\dex=\vcentcolon C_k.$$
Therefore, given $\phi\in C_{\mathrm{c}}(\Omega)$ such that $\mathrm{supp}(\phi)\subset K$, one has
$$\int\limits_\Omega\frac{f_n}{\big(u_n+\frac{1}{n}\big)^{\gamma}}\phi\dex\le C_K\|\phi\|_{L^\infty(K)},$$
whence $(f_n(u_n+1/n)^{-\gamma})$ is bounded in the space of Radon measures on $\Omega$. Since $(v_n^{1-\theta}u_n^{r-1})$ is equi-integrable, \eqref{ae convergence by Boccardo Murat 2} is a consequence Theorem \ref{Boccardo-Murat}, as before.

\emph{Step $7$.} Here, we show that
\begin{equation}\label{v sub-solution data dual space}
\int\limits_\Omega ADv\cdot D\psi\dex\le\int\limits_\Omega\frac{u^r}{v^\theta}\psi\dex\quad\forall\psi\in\HH(\Omega)\cap L^\infty(\Omega),\,\psi\ge0.
\end{equation}
Given $\varepsilon>0$ and $0\le\psi\in\HH(\Omega)\cap L^\infty(\Omega)$, we can choose $\frac{T_1(v_n)}{T_1(v)+\varepsilon}\psi$ as a test function. Let us stress that we need to truncate $v_n$ because we do not have any uniform $L^\infty$--bound. We have
\begin{multline}\label{approx v subsolution data dual space}
\int\limits_\Omega A D  v_n\cdot D  T_1(v_n)\frac{\psi}{T_1(v)+\varepsilon}-A D  v_n\cdot D  T_1(v)\frac{T_1(v_n)}{(T_1(v)+\varepsilon)^2}\psi\dex+\\
+\int\limits_\Omega A D  v_n\cdot D \psi\frac{T_1(v_n)}{T_1(v)+\varepsilon}\dex=\int\limits_\Omega \frac{u_n^r}{\big(v_n+\frac{1}{n}\big)^{\theta}}\frac{T_1(v_n)}{T_1(v)+\varepsilon}\psi\dex.
\end{multline}
We would like to let $n\to\infty$ in \eqref{approx v subsolution data dual space}. First, note that
\begin{multline*}
\int\limits_\Omega A D  v_n\cdot D  T_1(v_n)\frac{\psi}{T_1(v)+\varepsilon}\dex=\int\limits_{\{v_n\le 1\}} A D  v_n\cdot D  v\frac{\psi}{T_1(v)+\varepsilon}\dex+\\+\int\limits_{\{v_n\le 1\}}A( D  v_n- D  v)\cdot( D  v_n- D  v)\frac{\psi}{T_1(v)+\varepsilon}\dex+\\+\int\limits_{\{v_n\le 1\}}A D  v\cdot( D  v_n- D  v)\frac{\psi}{T_1(v)+\varepsilon}\dex=\alpha_{n}+\beta_{n}+\gamma_{n}.
\end{multline*}
Since, as a straightforward computation proves,
\begin{equation}\label{boundedness A DT_1(v_n)}
\| A D  T_1(v_n)\|_{L^2}\le\beta \|  D  v_n\|_{L^2}\le C\quad\forall n\in\N,
\end{equation}
by \eqref{ae convergence by Boccardo Murat}, \eqref{boundedness A DT_1(v_n)}, and the \Ae convergence of $(v_n)$ to $v$, it follows that $A D  T_1(v_n)\rightharpoonup A D  T_1(v)$ in $L^2(\Omega),$ which in turn implies
$$\alpha_n\to\int\limits_\Omega A D  T_1(v)\cdot D  v\frac{\psi}{T_1(v)+\varepsilon}\dex.$$
Now, noting that $\beta_n\ge0$ and
$$\gamma_n=\alpha_n-\int\limits_{\{v_n\le 1\}}A D  v\cdot D  v\frac{\psi}{T_1(v)+\varepsilon}\dex\to0,$$
we infer
\begin{equation}\label{liminf second equation data dual space}
\liminf_{n\to\infty}\int\limits_\Omega A D  v_n\cdot D  T_1(v_n)\frac{\psi}{T_1(v)+\varepsilon}\dex\ge \int\limits_\Omega A D  v\cdot D  T_1(v)\frac{\psi}{T_1(v)+\varepsilon}\dex.
\end{equation}
Therefore, thanks to \eqref{liminf second equation data dual space} and \eqref{conv L^1 second eq data dual space}, the weak convergence of ($v_n)$ to $v$, the \text{a.e.} convergence of $(u_n)$, $(v_n)$, and $(Dv_n)$ to $u$, $v$, and $Dv$, respectively, and the dominated convergence theorem, letting $n\to\infty$ in \eqref{approx v subsolution data dual space} yields
\begin{multline}\label{final step v subsol data in dual space}
\int\limits_\Omega A D v\cdot D  T_1(v)\frac{\psi}{T_1(v)+\varepsilon}-A D  v\cdot D  T_1(v)\frac{T_1(v_n)\psi}{(T_1(v)+\varepsilon)^2}\dex\\
+\int\limits_\Omega A D  v\cdot D \psi\frac{T_1(v)}{T_1(v)+\varepsilon}\dex\le\int\limits_\Omega \frac{u^r}{v^{\theta}}\frac{T_1(v)\psi}{T_1(v)+\varepsilon}\dex.
\end{multline}
Since the first integral on the l\text{.}h\text{.}s\text{.} in \eqref{final step v subsol data in dual space} is positive, and $\frac{T_1(v)}{T_1(v)+\varepsilon}\le1$ for all $\varepsilon>0$, using Fatou's lemma w\text{.}r\text{.}t\text{.} $\varepsilon\to0$ formula \eqref{v sub-solution data dual space} follows.

\emph{Step $8$.} Now we prove that
\begin{equation}\label{u sub-solution data dual space}
\int\limits_\Omega A D  u\cdot D \varphi+ v^{1-\theta} u^{r-1}\varphi\dex\le\int\limits_\Omega\frac{f}{u^\gamma}\varphi\dex\quad\forall\varphi\in\HH(\Omega)\cap L^\infty(\Omega),\,\varphi\ge0.
\end{equation}
Let $\varepsilon>0$ and $0\le\varphi\in\HH(\Omega)\cap L^\infty(\Omega)$ and define
$$H_{n,\varepsilon}\Def\frac{T_1(G_1(u_n))}{T_1(G_1(u))+\varepsilon}\varphi,\quad H_\varepsilon\Def\frac{T_1(G_1(u))}{T_1(G_1(u))+\varepsilon}\varphi.$$
Since $u_n\to u$ a\text{.}e\text{.},
$$\lim_{n\to\infty}\int\limits_\Omega \frac{f_n}{\big(u_n+\frac{1}{n}\big)^\gamma}H_{n,\varepsilon}\dex=\int\limits_\Omega \frac{f}{u^\gamma}H_\varepsilon\dex.$$
Moreover, by $\|H_{n,\varepsilon}\|_{L^\infty}\le\varepsilon^{-1}\|\varphi\|_{L^\infty}$ and \eqref{conv L1 first equation data in dual space}, extracting a subsequence if necessary, we have
$$v_n^{1-\theta} u_n^{r-1}H_{n,\varepsilon}\to v^{1-\theta} u^{r-1}H_\varepsilon\quad\text{strongly in}\ L^1(\Omega).$$
Then, \eqref{u sub-solution data dual space} follows taking $H_{n,\varepsilon}$ as a test function in \eqref{appr eq1} and arguing as in Step $6$, with \eqref{ae convergence by Boccardo Murat} replaced by \eqref{ae convergence by Boccardo Murat 2}.

\emph{Step $9$: conclusion.} By \eqref{u super-solution data dual space}, \eqref{u sub-solution data dual space}, \eqref{v super-solution data dual space}, and \eqref{v sub-solution data dual space}, we have
\begin{equation*}
\int\limits_\Omega A D  u\cdot D \varphi+v^{1-\theta} u^{r-1}\varphi\dex=\int\limits_\Omega\frac{f}{u^\gamma}\varphi\dex\quad\forall\varphi\in\HH(\Omega)\cap L^\infty(\Omega)
\end{equation*}
and
\begin{equation*}
\int\limits_\Omega A D  v\cdot D \psi\dex=\int\limits_\Omega \frac{u^r}{v^{\theta}}\psi\dex\quad\forall\psi\in\HH(\Omega)\cap L^\infty(\Omega).
\end{equation*}
In fact, the last formula holds also for unbounded test functions. This follows letting $k\to\infty$ in the formula
$$\int\limits_\Omega A D  v\cdot D T_k(\psi)\dex=\int\limits_\Omega \frac{u^r}{v^{\theta}}T_k(\psi)\dex,$$
which holds for any $\psi\in\HH(\Omega)$ and any $k\ge0$. The proof is concluded.

\emph{Step $10$: the case $r=2$.} Thanks to the previous steps, $(u,v)$ is a finite energy solution to \eqref{second_system}. The assertion on the regularity of $v$ follows from Lemma \ref{a priori estimates on (vn)}, so it remains to prove that $u$ is strictly positive. By \eqref{u super-solution data dual space}, for all $\varphi\in\HH(\Omega)\cap L^\infty(\Omega)$, $\varphi\ge0$, it holds
$$\int\limits_\Omega ADu\cdot D\varphi\dex+\int\limits_\Omega v^{1-\theta}u\varphi\dex\ge0.$$
This remains true also for unbounded test functions, allowing the second integral to become infinite. Indeed, for $\varphi\in\HH(\Omega)$ and $k\in\N$,
$$\int\limits_\Omega ADu\cdot DT_k(\varphi)\dex+\int\limits_\Omega v^{1-\theta}uT_k(\varphi)\dex\ge0,$$
and the inequality for $\varphi$ follows letting $k\to\infty$ and exploiting the dominated and the monotone convergence theorems. Thence, $u$ is a non-negative, weak solution to
$$-\Div(A(x)D u)+v^{1-\theta}u\ge0.$$
Since $v^{1-\theta}\in L^{\frac{s_m}{1-\theta}}(\Omega)$, with $s_m$ given by \eqref{regularity v dual space}, and $u>0$ \Ae in $\{f>0\}$, by the weak Harnack inequality (Theorem $7.1.2$ in \cite{pucci-serrin}) $u$ is strictly positive in $\Omega$ if
\begin{equation}\label{condition sm}
s_m>\frac{d}{2}(1-\theta).
\end{equation}
We want to show that conditions \eqref{conditions strict positivity dual space} ensure that \eqref{condition sm} holds.
If $d=3,4$, then $\big(\frac{2^\star}{1-\gamma}\big)'\ge \frac{d}{3+\gamma}$, and \eqref{condition sm} holds thanks to the very definition of $s_m$ (see \eqref{regularity v dual space}). If $d\ge 5$, the only uncertain situation is $m<\frac{d}{3+\gamma}$, i.e., $s_m=\big(\frac{m^{\star\star}(1+\gamma)}{2}\big)^{\star\star}(1+\theta)$. Note that
\begin{equation}\label{almost last display data dual space}
\Big(\frac{m^{\star\star}(1+\gamma)}{2}\Big)^{\star\star}(1+\theta)>\frac{d}{2}(1-\theta)\quad\text{if and only if}\quad m>\frac{1-\theta}{2(2-\theta+\gamma)}d
\end{equation}
and that
\begin{equation}\label{last display data dual space}
\bigg(\frac{2^\star}{1-\gamma}\bigg)'>\frac{1-\theta}{2(2-\theta+\gamma)}d\quad\text{if and only if}\quad d<2\frac{3-\theta}{1-\theta}.
\end{equation}
Since the function $[0,1)\ni\theta\mapsto2\frac{3-\theta}{1-\theta}$ is increasing, its minimum in $[0,1)$ is $6$. Therefore, when $d=5$, by \eqref{almost last display data dual space} and \eqref{last display data dual space}, \eqref{condition sm} holds (recall that $m>\big(\frac{2^\star}{1-\gamma}\big)'$). Finally, if $d\ge 6$, the validity of \eqref{condition sm} follows by the conditions in \eqref{conditions strict positivity dual space}, which match \eqref{almost last display data dual space} and \eqref{last display data dual space}.
\end{proof}
\section{Data outside the dual space}\label{data outside dual space}
Here, we assume $m<\big(\frac{2^\star}{1-\gamma}\big)'$. Let $(u_n,v_n)\in\big(\HH(\Omega)\cap L^\infty(\Omega)\big)^2$ be the sequence given by Lemma \ref{approximation lemma}.
\subsection{A priori estimates} Repeating the reasoning that lead to \eqref{holder dual data} we obtain
\begin{equation}\label{energy bound step data outside dual space}
\alpha\int\limits_\Omega | D  u_n|^2\dex+\alpha\int\limits_\Omega| D  v_n|^2\dex\le \int\limits_\Omega f_n u_n^{1-\gamma}\dex.
\end{equation}
\begin{lemma}\label{bound in Lr}
Let $m\ge\big(\frac{r}{1-\gamma}\big)'$ and $r>2^\star$. Then $(u_n)$ is bounded in $L^{r}(\Omega)$.
\end{lemma}
\begin{proof}
Observe that
$$\int\limits_\Omega u_n^r\dex=\int\limits_{\{v_n\ge1\}}u_n^r\dex+\int\limits_{\{v_n\le1\}}u_n^r\dex.$$
Moreover,
\begin{equation}\label{bound un to r on vn ge 1}
\int\limits_{\{v_n\ge1\}}u_n^r\dex\le \int\limits_{\{v_n\ge1\}}v_n^{1-\theta} u_n^r\dex\le \int\limits_\Omega f_n u_n^{1-\gamma}\dex
\end{equation}
and
\begin{align*}
\int\limits_{\{v_n\le1\}}u_n^r\dex &\le |\Omega|+\int\limits_{\{v_n\le1\}} u_n^{r+1}\dex\le |\Omega|+2^\theta\int\limits_{\{v_n\le1\}} \frac{u_n^{r+1}}{\big(v_n+\frac{1}{n}\big)^{\theta}}\dex,
\end{align*}
as a simple computation shows. Since, by \eqref{inequality for second equation},
\begin{align*}
\int\limits_{\{v_n\le1\}} \frac{u_n^{r+1}}{\big(v_n+\frac{1}{n}\big)^{\theta}}\dex\le\int\limits_\Omega A D  T_1(v_n)\cdot D  u_n\dex
\end{align*}
and, taking advantage of \eqref{energy bound step data outside dual space},
$$\int\limits_\Omega A D  T_1(v_n)\cdot D  u_n \dex\le\frac{\beta}{2}\int\limits_\Omega | D  v_n|^2\dex+\frac{\beta}{2}\int\limits_\Omega| D  u_n|^2\dex\le C\int\limits_\Omega f_n u_n^{1-\gamma}\dex,$$
we infer
\begin{equation}\label{bound un to r on vn le 1}
    \int\limits_{\{v_n\le1\}}u_n^r\dex \le |\Omega|+C\int\limits_\Omega f_n u_n^{1-\gamma}\dex.
\end{equation}
Putting together \eqref{bound un to r on vn ge 1} and \eqref{bound un to r on vn le 1} yields
$$\int\limits_\Omega u_n^r\dex\le |\Omega|+ C\int\limits_\Omega f_n u_n^{1-\gamma}\dex,$$
hence, recalling that $u_n\in L^\infty(\Omega)$ and that, by assumption, $m'(1-\gamma)\le r$,
\begin{equation*}
\|u_n\|_{L^r}^r\le |\Omega|+C \|f\|_{L^m} \|u_n\|_{L^r}^{1-\gamma}.
\end{equation*}
It follows that
$$\|u_n\|_{L^r}^r\big(1-C\|f\|_{L^m}\|u_n\|_{L^r}^{1-\gamma-r}\big)\le|\Omega|,$$
which, given that $1-\gamma-r<0$, entails the conclusion.
\end{proof}
\begin{lemma}\label{bound in L(r+1)}
Let $m\ge\big(\frac{r+1}{1-\gamma}\big)'$, $r>2^\star-1$, and ${\theta}=0$. Then $(u_n)$ is bounded in $L^{r+1}(\Omega)$.
\end{lemma}
\begin{proof}
Under the assumption ${\theta}=0$, the equations \eqref{appr eq1}, \eqref{appr eq2} become
$$-\Div(A(x) D  u_n)+v_n u_n^{r-1}=\frac{f_n}{\big(u_n+\frac{1}{n}\big)^\gamma}\,,\qquad-\Div(A(x) D  v_n)=u_n^r.$$
Choosing $u_n$ as test function in the second one produces
$$\int\limits_\Omega u_n^{r+1}\dex=\int\limits_\Omega A D  v_n\cdot D  u_n\dex\le \frac{\beta}{2} \int\limits_\Omega | D  v_n|^2\dex+ \frac{\beta}{2} \int\limits_\Omega | D  u_n|^2\dex,$$
whence, exploiting \eqref{energy bound step data outside dual space},
$$\int\limits_\Omega u_n^{r+1}\dex\le C\int\limits_\Omega f_n u_n^{1-\gamma}\dex.$$
Since $m'(1-\gamma)\le r+1$, one has
$$\int\limits_\Omega u_n^{r+1}\dex\le C\|f\|_{L^m} \|u_n\|_{L^{r+1}}^{1-\gamma}.$$
This implies $\|u_n\|_{L^{r+1}}\le C\|f\|_{L^m}^{\frac{1}{r+\gamma}}$, as desired.
\end{proof}
\subsection{Proofs of the results} Our proofs follow the same line of that of Theorem \ref{existence data in dual space}. We just sketch the proof of Theorem \ref{regularizing outside dual space}.
\begin{proof}[\textbf{Proof of Theorem \ref{regularizing outside dual space}}]
Since $m'(1-\gamma)\le r$, by H\"older's inequality
$$\int\limits_\Omega f_n u_n^{1-\gamma}\dex\le |\Omega|^{\frac{1}{m'}-\frac{1-\gamma}{r}}\|f\|_{L^m}\|u_n\|_{L^r}^{1-\gamma}.$$
Then, by Lemma \ref{bound in Lr}, $(f_n u_n^{1-\gamma})$ is bounded in $L^1$ (and actually equi-integrable). Then, more solito, plugging the latter information into \eqref{energy bound step data outside dual space} implies that there exist $u,v\in\HH(\Omega)$ such that, up to subsequences, $(u_n)$ and $(v_n)$ converge, respectively, to $u$ and $v$ weakly in $\HH(\Omega)$, strongly in $L^p(\Omega)$ for all $1\le p<2^\star$, and \text{a.e.} in $\Omega$. Then, repeating the arguments in Theorem \eqref{existence data in dual space}, the proof proceeds as follows:
\begin{itemize}
    \item[\emph{Step $1$}.] By the equi-integrability of $(f_n u_n^{1-\gamma})$, $v_n^{1-\theta}u_n^{r-1}\to v^{1-\theta}u^{r-1}$ strongly in $L^1(\Omega)$.
    \item[\emph{Step $2$}.] By Fatou's lemma and Step $1$, $u>0$ \Ae in $\{f>0\}$, $v>0$ \Ae in $\Omega$, and they are super-solutions.
    \item[\emph{Step $3$}.] The same as Step $3$ of Theorem \ref{existence data in dual space}.
    \item[\emph{Step $4$}.] $Dv_n\to Dv$ and $Du_n\to Du$ \Ae in $\Omega$.
    \item[\emph{Step $5$}.] $u$ and $v$ are solutions.
\end{itemize}
Finally, Fatou's lemma entails
$$\|u\|_{L^r}\le \liminf_{n\to\infty}\|u_n\|_{L^r}\le C,$$
i.e., $u\in L^r(\Omega)$.
\end{proof}
\begin{proof}[\textbf{Proof of Theorem \ref{regularizing outside dual space 2}}]
Argue as in the previous proof.
\end{proof}
\section{A higher integrability result}\label{A higher integrability result}
Assume
$ (r+\gamma)' \le m<\frac{d}{2}$ and $r\ge\frac{d}{d-2}-\gamma$ and let $(u_n,v_n)$ be as in Lemma \ref{approximation lemma}. Since
$$(r+\gamma)' \ge \frac{r}{r-1+\gamma}=\Big(\frac{r}{1-\gamma}\Big)',$$
the proof of Theorem \ref{regularizing outside dual space} shows that $(u_n,v_n)$ converges almost everywhere to a finite-energy solution $(u,v)\in\HH(\Omega)\times\HH(\Omega)$. Then, Theorem \ref{last thm} is a straightforward consequence of the next
\begin{lemma}
There exists a positive constant $C$, depending on $|\Omega|$, $\gamma$, $r$, $\alpha$, and $\beta$, such that
$$\int\limits_\Omega u_n^{r+1+\gamma}\dex\le C \|f\|_{L^m}\Big(1+\|f\|_{L^m}^{\frac{1}{r-1+\gamma}}\Big)\quad \forall n\in\N.$$
\end{lemma}
\begin{proof}
We have
\begin{equation}\label{splitting last theorem}
\int\limits_\Omega u_n^{r+1+\gamma}\dex=\int\limits_{\{v_n\le1\}} u_n^{r+1+\gamma}\dex+\int\limits_{\{v_n\ge1\}} u_n^{r+1+\gamma}\dex.
\end{equation}
Let us treat the two integrals in the above formula separately. Concerning the fist one, by \eqref{inequality for second equation} we have
\begin{align}\label{step first bound last theorem}
\int\limits_{\{v_n\le1\}} u_n^{r+1+\gamma}\dex\le 2^\theta\int\limits_{\{v_n\le1\}} \frac{u_n^{r+1+\gamma}}{\big(v_n+\frac{1}{n}\big)^\theta}\dex\le 2^\theta\int\limits_\Omega A D  T_1(v_n)\cdot  D  u_n (\gamma+1) u_n^\gamma\dex.
\end{align}
Pick $\varepsilon>0$ and choose $T_1(v_n) (u_n+\varepsilon)^\gamma$ as test function in \eqref{appr eq1}. Disregarding the positive terms
$$\gamma\int\limits_\Omega A D  u_n\cdot  D  u_n T_1(v_n) (u_n+\varepsilon)^{\gamma-1}\dex\quad\text{and}\quad\int\limits_\Omega v_n^{1-\theta} u_n^{r-1}T_1(v_n) (u_n+\varepsilon)^{\gamma}\dex,$$
we obtain
$$\gamma\int\limits_\Omega A D  u_n\cdot  D  T_1(v_n)(u_n+\varepsilon)^\gamma\dex\le\int\limits_\Omega \frac{f}{\big(u_n+\frac{1}{n}\big)^\gamma}(u_n+\varepsilon)^\gamma\dex.$$
Letting $\varepsilon\to0$ produces
$$\gamma\int\limits_\Omega A D  u_n\cdot  D  T_1(v_n)u_n^\gamma\dex\le\int\limits_\Omega f\dex,$$
which, together with \eqref{step first bound last theorem}, imply
\begin{equation}\label{first estimate last theorem}
\int\limits_{\{v_n\le1\}} u_n^{r+1+\gamma}\dex\le C\int\limits_\Omega f\dex\le C\|f\|_{L^m}.
\end{equation}
Now we estimate the second integral in \eqref{splitting last theorem}. Testing \eqref{appr eq1} against $u_n^{\gamma+2}$, one has
\begin{equation}\label{step second estimate last theorem}
\int\limits_{\{v_n\ge1\}}u_n^{r+1+\gamma}\dex\le \int\limits_\Omega v_n^{1-\theta} u_n^{r+1+\gamma}\dex\le\int\limits_\Omega f u_n^2\dex.
\end{equation}
Set $E\Def\big\{f\le \medmath{\frac{1}{2}}u_n^{r-1+\gamma}\big\}$, we obtain
$$\int\limits_\Omega f u_n^2\dex=\int\limits_E f u_n^2\dex + \int\limits_{\Omega\setminus E} f u_n^2\dex\le \int\limits_\Omega \frac{1}{2}u_n^{r+1+\gamma}+Cf^{ (r+\gamma)' }\dex.$$
Thanks to our (key) assumption $m\ge (r+\gamma)' $,
$$\bigg(\int\limits_\Omega f^{ (r+\gamma)' }\dex\bigg)^{\frac{1}{(r+\gamma)'}}\le C\|f\|_{L^m},$$
hence
$$\int\limits_\Omega f u_n^2\dex\le \frac{1}{2}\int\limits_\Omega u_n^{r+1+\gamma}\dex+C\|f\|_{L^m}^{ (r+\gamma)' },$$
and by \eqref{step second estimate last theorem}
\begin{equation}\label{second estimate last theorem}
\int\limits_{\{v_n\ge1\}}u_n^{r+1+\gamma}\dex\le\frac{1}{2}\int\limits_\Omega u_n^{r+1+\gamma}\dex+C\|f\|_{L^m}^{ (r+\gamma)' }.
\end{equation}
Finally, putting \eqref{splitting last theorem}, \eqref{first estimate last theorem} and \eqref{second estimate last theorem} together we obtain
$$\int\limits_\Omega u_n^{r+1+\gamma}\dex\le C\|f\|_{L^m}+\frac{1}{2}\int\limits_\Omega u_n^{r+1+\gamma}\dex+C\|f\|_{L^m}^{ (r+\gamma)' },$$
and the conclusion follows.
\end{proof}
\section*{Acknowledgments}
The author is a member of GNAMPA of INdAM and is supported by University of Florence, University of Perugia, and INdAM. This paper is based on the author master thesis, that he wrote as a student at University of Catania. The author warmly thanks Salvatore Angelo Marano and Sunra Mosconi for many stimulating discussions, and for their valuable comments about the content of the paper and its organization. Finally, the author would also like to acknowledge Marco Picerni for a fruitful discussion regarding the positivity of solutions.
\printbibliography

@article{DeCaveOlivaStrani2016,
author = {De Cave, L. M. and Oliva, F. and Strani, M.},
title = {Existence of solutions to a non-variational singular elliptic system with unbounded weights},
journal = {Math. Nachr.},
volume = {290},
number = {2-3},
pages = {236-247},
doi = {10.1002/mana.201600038},
year = {2017}
}

@article{Oliva2019,
title = {Regularizing effect of absorption terms in singular problems},
journal = {J. Math. Anal. Appl.},
volume = {472},
number = {1},
pages = {1136-1166},
year = {2019},
doi = {10.1016/j.jmaa.2018.11.069},
author = {Oliva F.}
}

@article{OlivaPetitta2018,
title = {Finite and infinite energy solutions of singular elliptic problems: Existence and uniqueness},
journal = {J. Differential Equations},
volume = {264},
number = {1},
pages = {311-340},
year = {2018},
doi = {10.1016/j.jde.2017.09.008},
author = {F. Oliva and F. Petitta}
}

@book{adamsfournier,
    author = {Adams, R. A. and Fournier, J. J. F.},
    title = {Sobolev Spaces},
    publisher = {Academic Press},
    year = {2003},
    edition ={2},
    series = {Pure and Applied Mathematics},
    volume = {140}
}

@article{stampacchia65,
     author = {Stampacchia, G.},
     journal = {Ann. Inst. Fourier (Grenoble)},
     title={Le problème de Dirichlet pour les équations elliptiques du second ordre à coefficients discontinus},
     pages = {189--257},
     publisher = {Institut Fourier},
     address = {Grenoble},
     volume = {15},
     number = {1},
     year = {1965},
     doi = {10.5802/aif.204}
}

@article{Boccardomurat1982,
author = {Boccardo, L. and Murat, F.},
journal = {Portugaliae mathematica},
number = {1-4},
pages = {535-562},
publisher = {Sociedade Portuguesa de Matemática},
title = {Remarques sur l'homogénéisation de certains problèmes quasi-linéaires},
volume = {41},
year = {1982}
}

@book{pucci-serrin,
author = {P. Pucci and J. Serrin},
title = {The maximum principle},
publisher = {Birkh\"auser Verlag},
series = {Progress in Nonlinear Differential Equations and Their Applications},
volume = {73},
year = {2007},
doi= {10.1007/978-3-7643-8145-5}
}

@article{BoccardoOrsina2016,
author = {Boccardo, L. and Orsina, L.},
title = {Regularizing effect for a system of {S}chrödinger-{M}axwell equations},
year = {2016},
volume = {11},
journal = {Adv. Calc. Var.},
doi = {10.1515/acv-2016-0006}
}

@article{Durastanti_2019,
   title={Regularizing effect for some $p$-Laplacian systems},
   volume={188},
   doi={10.1016/j.na.2019.06.011},
   journal={Nonlinear Anal.},
   publisher={Elsevier BV},
   author={Durastanti, R.},
   year={2019},
   pages={425–438}
}

@article{guarnotta2022,
author = {Guarnotta, U. and Livrea, R. and Marano, S. A.},
year = {2022},
pages = {416-428},
title = {Some recent results on singular p-Laplacian equations},
volume = {55},
journal = {Demonstr. Math.},
doi = {10.1515/dema-2022-0031} 
}

@article{BoccardoOrsina2018,
author = {Boccardo, L. and Orsina, L.},
title = {Regularizing effect of the lower order terms in some elliptic problems: old and new},
year = {2018},
volume = {29},
number={2},
journal = {Atti Accad. Naz. Lincei Cl. Sci. Fis. Mat. Natur.},
pages={387-399},
doi = {10.4171/RLM/812}
}

@article{Boccardo2016SpiritofBenciandFortunato,
author = {Boccardo, L.},
year = {2016},
title = {Elliptic systems of {S}chr\"odinger type in the spirit of {B}enci-{F}ortunato},
pages = {321-331},
volume = {15},
journal = {Adv. Nonlinear Stud.},
doi = {10.1515/ans-2015-0203}
}

@article{LiuMosconi2020689,
title = {On the {S}chr\"odinger-{P}oisson system with indefinite potential and 3-sublinear nonlinearity},
journal = {J. Differential Equations},
volume = {269},
number = {1},
pages = {689-712},
year = {2020},
doi = {10.1016/j.jde.2019.12.023},
author = {Liu, S. and Mosconi, S. J. N.}
}

@article{BOCCARDO2022126490,
title = {An elliptic system with singular nonlinearities: Existence via non variational arguments},
journal = {J. Math. Appl.},
volume = {516},
number = {1},
pages = {126490},
year = {2022},
author = {L. Boccardo and S. Buccheri and C. A. {dos Santos}},
doi={10.1016/j.jmaa.2022.126490}
}

@article{stampacchia6364,
     author = {Stampacchia, G.},
     title = {Équations elliptiques du second ordre \`a coefficients discontinus},
     journal = {S\'eminaire Jean Leray},
     pages = {1--77},
     publisher = {Coll\`ege de France},
     number = {3},
     year = {1963}
}

@book{evans-gariepy,
  title={Measure Theory and Fine Properties of Functions, Revised Edition},
  author={Evans, L.C. and Gariepy, R.F.},
  year={2015},
  edition ={1},
  publisher={Chapman and Hall/CRC},
  doi={10.1201/b18333}
}

@article{BOCCARDO1992581,
title = {Almost everywhere convergence of the gradients of solutions to elliptic and parabolic equations},
journal = {Nonlinear Analysis: Theory, Methods \& Applications},
volume = {19},
number = {6},
pages = {581-597},
year = {1992},
author = {L. Boccardo and F. Murat},
doi={10.1016/0362-546X(92)90023-8}
}

@book{gilbarg2001elliptic,
  title={Elliptic Partial Differential Equations of Second Order},
  author={Gilbarg D. and Trudinger N. S.},
  series={Classics in Mathematics},
  year={2001},
  publisher={Springer Berlin Heidelberg}
}

@article{oliva2024singularellipticpdesextensive,
author = {Oliva, F. and Petitta, F.},
title={Singular elliptic {PDE}s: an extensive overview},
year= {2025},
journal={ Partial Differ. Equ. Appl.},
specialissue={6},
volume={6},
doi={10.1007/s42985-024-00308-9}
}

@article{BoccardoOrsina2010,
author = {Boccardo, L. and Orsina, L.},
title = {Semilinear elliptic equations with singular nonlinearities},
year = {2010},
pages = {363-380},
volume = {37},
journal = {Calc. Var.},
doi = {10.1007/s00526-009-0266-x}
}

@article{BoccardoOrsina2011,
title = {A variational semilinear singular system},
journal = {Nonlinear Anal.},
volume = {74},
number = {12},
pages = {3849-3860},
year = {2011},
doi = {10.1016/j.na.2011.01.017},
author = {L. Boccardo and L. Orsina}
}

@article{Trudinger1967,
author = {Trudinger, N. S.},
title = {On {H}arnack type inequalities and their application to quasilinear elliptic equations},
journal = {Comm. Pure Appl. Math.},
volume = {20},
number = {4},
pages = {721-747},
doi ={10.1002/cpa.3160200406},
year = {1967}
}

@article{Benci_Fortunato_1998, title={An eigenvalue problem for the {S}chr\"odinger-{M}axwell equations}, volume={11}, number={2}, journal={Topol. Methods Nonlinear Anal.}, author={Benci, V. and Fortunato, D.}, year={1998}, pages={283–293}, doi={10.12775/TMNA.1998.019} }

@article{Guarnotta2023,
title = {Some recent results on singular $ p $-Laplacian systems},
journal = {Discrete Contin. Dyn. Syst. Ser. S},
volume = {16},
number = {6},
pages = {1435-1451},
year = {2023},
doi = {10.3934/dcdss.2022170},
author = {Guarnotta, U. and Livrea, R. and Marano, S. A.}
}

@article{BoccardoOrsina2024,
  title={A Singular System of {S}chrödinger-{M}axwell Equations},
  author={Boccardo, L and Orsina, L.},
  year={2024},
  volume = {21},
  journal = {Mediterr. J. Math.},
  doi = {10.1007/s00009-024-02632-1},
  number={94},
}

@article{Lazer1991OnAS,
  author={Lazer, A. C. and McKenna, P. J.},
  title={On a singular nonlinear elliptic boundary-value problem},
  year={1991},
  volume = {111},
  journal = {Comm. Partial Differential Equations},
  doi = {10.1090/S0002-9939-1991-1037213-9},
  pages ={721-730},
  number={3},
}
\end{document}